\documentclass[11pt]{article}

\usepackage{amsmath}

\usepackage{latexsym}
\usepackage{amssymb}
\usepackage{amsfonts}\large
\usepackage{amsthm}
\usepackage{fontenc}
\usepackage{fullpage}
\usepackage{enumerate}
\usepackage{esint}
\usepackage{bm}
\usepackage{hyperref}
\usepackage{titling}
\usepackage{indentfirst}

\hypersetup{colorlinks=true}
\usepackage{xcolor}
\usepackage[normalem]{ulem}
\usepackage{marginnote}
\usepackage{enumerate}

\DeclareMathOperator{\dive}{div}

\DeclareMathOperator{\e}{\varepsilon}

\DeclareMathOperator{\loc}{loc}
\DeclareMathOperator{\bR}{\mathbb R}

\newtheorem{thm}{Theorem}[section]
\newtheorem{lemma}[thm]{Lemma}
\newtheorem{prop}[thm]{Proposition}

\theoremstyle{definition}
\newtheorem*{acknowledgments*}{Acknowledgments}

\theoremstyle{definition}
\newtheorem{remark}[thm]{Remark}

\renewcommand\footnotemark{}

\makeatletter
\newcommand{\@endstuff}{\par\vspace{\baselineskip}\noindent\small
\begin{tabular}{@{}l}\scshape{Department of Mathematics, Universitat Aut{\`o}noma de Barcelona,}\\\scshape{08193 Bellaterra (Barcelona), Spain}\\ \\\textit{E-mail address}: \texttt{gsakellaris@mat.uab.cat} \end{tabular}}
\AtEndDocument{\@endstuff}
\makeatother

\numberwithin{equation}{section}

\begin{document}
\title{Scale invariant regularity estimates for second order elliptic equations with lower order coefficients in optimal spaces}

\author{Georgios Sakellaris \thanks{\hspace*{-7pt}2010 \textit{Mathematics Subject Classification}. Primary 35B45, 35B50, 35B51, 35B65, 35J15. Secondary 35D30, 35J10, 35J20, 35J86. \newline
\hspace*{10.5pt} \textit{Key words and phrases}. Maximum principle; pointwise bounds; Moser estimate; Harnack inequality; continuity of solutions; Lorentz spaces; decreasing rearrangements; symmetrization. \newline
\hspace*{14pt}The author has received funding from the European Union's Horizon 2020 research and innovation programme under Marie Sk{\l}odowska-Curie grant agreement No 665919, and is partially supported by MTM-2016-77635-P (MICINN, Spain) and 2017 SGR 395 (Generalitat de Catalunya).}}

\date{\vspace{-7ex}}
\maketitle

\begin{abstract}
We show local and global scale invariant regularity estimates for subsolutions and supersolutions to the equation $-\dive(A\nabla u+bu)+c\nabla u+du=-\dive f+g$, assuming that $A$ is elliptic and bounded. In the setting of Lorentz spaces, under the assumptions $b,f\in L^{n,1}$, $d,g\in L^{\frac{n}{2},1}$ and $c\in L^{n,q}$ for $q\leq\infty$, we show that, with the surprising exception of the reverse Moser estimate, scale invariant estimates with ``good" constants (that is, depending only on the norms of the coefficients) do not hold in general. On the other hand, assuming a necessary smallness condition on $b,d$ or $c,d$, we show a maximum principle and Moser's estimate for subsolutions with ``good" constants. We also show the reverse Moser estimate for nonnegative supersolutions with ``good" constants, under no smallness assumptions when $q<\infty$, leading to the Harnack inequality for nonnegative solutions and local continuity of solutions. Finally, we show that, in the setting of Lorentz spaces, our assumptions are the sharp ones to guarantee these estimates.
\end{abstract}

\section{Introduction}

In this article we are interested in local and global regularity for subsolutions and supersolutions to the equation $\mathcal{L}u=-\dive f+g$, in domains $\Omega\subseteq\bR^n$, where $\mathcal{L}$ is of the form
\[
\mathcal{L}u=-\dive(A\nabla u+bu)+c\nabla u+du.
\]
In particular, we investigate the validity of the maximum principle, Moser's estimate, the Harnack inequality and continuity of solutions, in a scale invariant setting; that is, we want our estimates to not depend on the size of $\Omega$. We will also assume throughout this article that $n\geq 3$.

In this work $A$ will be bounded and uniformly elliptic in $\Omega$: for some $\lambda>0$,
\[
\left<A(x)\xi,\xi\right>\geq\lambda\|\xi\|^2,\quad \forall x\in\Omega,\,\,\,\forall\xi\in\bR^n.
\]
For the lower order coefficients and the terms on the right hand side, we consider Lorentz spaces that are scale invariant under the natural scaling for the equation. That is, we assume that
\[
b,f\in L^{n,1}(\Omega),\quad c\in L^{n,q}(\Omega),\quad d,g\in L^{\frac{n}{2},1}(\Omega),\quad q\leq\infty.
\]
In the case that $q=\infty$, it is also necessary to assume that the norm of $c$ is small for our results to hold. As explained in Section~\ref{secOptimality}, these assumptions are the optimal ones to imply our estimates in the setting of Lorentz spaces. Note also that there will be no size assumption on $\Omega$ and no regularity assumption on $\partial\Omega$. 

The main inspiration for this work comes from the local and global pointwise estimates for subsolutions to the fore mentioned operator in \cite{SakAPDE}, where it is also assumed that $d\geq\dive c$ in the sense of distributions. Focusing on the case when $c,d\equiv 0$ for simplicity, and assuming that $b\in L^{n,1}$, a maximum principle for subsolutions to $-\dive(A\nabla u+bu)\leq-\dive f+g$ is shown in \cite[Proposition 7.5]{SakAPDE}, while a Moser type estimate is the context of \cite[Proposition 7.8]{SakAPDE}. The main feature of these estimates is their scale invariance, with constants that depend only on the ellipticity of $A$ and the $L^{n,1}$ norm of $b$, as well as the $L^{\infty}$ norm of $A$ for the Moser estimate.

Following this line of thought, it could be expected that the consideration of all the lower order coefficients in the definition of $\mathcal{L}$ should yield the same type of scale invariant estimates, with constants being ``good"; that is, depending only on $n$, $q$, the ellipticity of $A$, and the norms of the coefficients involved (as well as $\|A\|_{\infty}$ in some cases). However, it turns out that this does not hold. In particular, if $B_1$ is the unit ball in $\bR^n$, in Proposition~\ref{dShouldBeSmall} we construct a bounded sequence $(d_N)$ in $L^{\frac{n}{2},1}(B_1)$ and a sequence $(u_N)$ of nonnegative $W_0^{1,2}(B_1)$ solutions to the equation $-\Delta u_N+d_Nu_N=0$ in $B_1$, such that
\[
\|u_N\|_{W_0^{1,2}(B_1)}\leq C,\quad\text{while}\quad\|u_N\|_{L^{\infty}(B_{1/2})}\xrightarrow[N\to\infty]{}\infty.
\]
We also show in Remark~\ref{bcShouldBeSmall} that the equation $-\Delta u-\dive(bu)+c\nabla u=0$ has the same feature, which implies that the constants in Moser's local boundedness estimate, as well as the Harnack inequality, cannot be ``good" without any further assumptions.

Since scale invariant estimates with ``good" constants do not hold in such generality, we first prove estimates where the constants are allowed to depend on the coefficients themselves. This is the context of the global bound in Proposition~\ref{GlobalUpperBound}, where it is shown that, if $\Omega\subseteq\bR^n$ is a domain and $u\in Y^{1,2}(\Omega)$ (see \eqref{eq:Y}) is a subsolution to $\mathcal{L}u\leq -\dive f+g$, then, for any $p>0$,
\begin{equation}\label{eq:maxPrinc}
\sup_{\Omega}u^+\leq C\sup_{\partial\Omega}u^++C'\left(\int_{\Omega}|u^+|^p\right)^{\frac{1}{p}}+C\|f\|_{n,1}+C\|g\|_{\frac{n}{2},1},
\end{equation}
where $C$ is a ``good" constant, while $C'$ depends on the coefficients themselves and $p$. Note the appearance of a constant in front of the term $\sup_{\partial\Omega}u^+$; such a constant can be greater than $1$, and this follows from the fact that constants are not necessarily subsolutions to our equation in the generality of our assumptions.

Having proven the previous estimate, we then turn to show various scale invariant estimates with ``good" constants, assuming an extra condition on the lower order coefficients, which is necessary in view of the fore mentioned discussion. Such a condition is some type of smallness: in particular, we either assume that the norms of $b,d$ are small, or that the norms of $c,d$ are small. Under these smallness assumptions, we show in Propositions~\ref{MaxPrincipleC} and \ref{MaxPrincipleB} that we can take $C'=0$ in \eqref{eq:maxPrinc}, leading to a maximum principle, and the Moser estimate for subsolutions to $\mathcal{L}u\leq -\dive f+g$ is shown in Propositions~\ref{MoserB} and \ref{MoserC}; that is, in the case when $b,d$ are small, or $c,d$ are small, then for any $p>0$,
\begin{equation}\label{eq:up}
\sup_{B_r}u\leq C\left(\fint_{B_{2r}}|u^+|^p\right)^{\frac{1}{p}}+C\|f\|_{L^{n,1}(B_{2r})}+C\|g\|_{L^{\frac{n}{2},1}(B_{2r})},
\end{equation}
where the constant $C$ is ``good", and also depends on $p$. In addition, the analogous estimate close to the boundary is deduced in Propositions~\ref{MoserBBoundary} and \ref{MoserCBoundary}.

On the other hand, somewhat surprisingly, we discover that even if the scale invariant Moser estimate with ``good" constants requires some type of smallness, it turns out that the scale invariant reverse Moser estimate with ``good" constants holds in the full generality of our initial assumptions. That is, in Proposition~\ref{lowerBound}, we show that if $u\in W^{1,2}(B_{2r})$ is a nonnegative supersolution to $\mathcal{L}u\geq-\dive f+g$, and under no smallness assumptions (when $q<\infty$), then for some $\alpha=\alpha_n$,
\begin{equation}\label{eq:low}
\left(\fint_{B_r}u^{\alpha}\right)^{\frac{1}{\alpha}}\leq C\inf_{B_{r/2}}u+C\|f\|_{L^{n,1}(B_{2r})}+C\|g\|_{L^{\frac{n}{2},1}(B_{2r})},
\end{equation}
where $C$ is a ``good" constant. Moreover, the analogue of this estimate close to the boundary is deduced in Proposition~\ref{lowerBoundBoundary}. Then, the Harnack inequality (Theorems~\ref{HarnackForB} and \ref{HarnackForC}) and continuity of solutions (Theorems~\ref{ContinuityB} and \ref{ContinuityC}) are shown combining \eqref{eq:up} and \eqref{eq:low}; for those, in order to obtain estimates with ``good" constants, it is again necessary to assume a smallness condition. Finally, having shown the previous estimates, we also obtain their analogues in the generality of our initial assumptions, with constants that depend on the coefficients themselves (Remarks~\ref{noSmallness}, \ref{noSmallness2} and \ref{noSmallness3}).

As a special case, we remark that all the scale invariant estimates above hold, with ``good" constants, in the case of the operators
\[
\mathcal{L}_1u=-\dive(A\nabla u)+c\nabla u,\qquad \mathcal{L}_2u=-\dive(A\nabla u+bu),
\]
under no smallness assumptions when $b\in L^{n,1}$ and $c\in L^{n,q}$, $q<\infty$.

\subsection*{The techniques}

The assumption that the coefficients $b,d$ lie in scale invariant spaces is reflected in the fact that the classical method of Moser iteration does not seem to work in this setting. More specifically, an assumption of the form $b\in L^{n,q}$ for some $q>1$ does not necessarily guarantee pointwise upper bounds (see Remark~\ref{optimalB}), and it is necessary to assume that $b\in L^{n,1}$ in order to deduce these bounds. However, Moser's method does not seem to be ``sensitive" enough to distinguish between the cases $b\in L^{n,1}$ and $b\in L^{n,q}$ for $q>1$. Thus, a procedure more closely related to Lorentz spaces has to be followed, and the first results in this article (Section~\ref{secGlobal}) are based on a symmetrization technique, leading to estimates for decreasing rearrangements. This technique involves a specific choice of test functions and has been used in the past by many authors, going back to Talenti's article \cite{TalentiElliptic}; here we use a slightly different choice, utilized by Cianchi and Mazya in \cite{CianchiMazyaSchrodinger}. However, since all the lower order coefficients are present, our estimates are more complicated, and we have to rely on an argument using Gr{\"o}nwall's inequality (as in \cite{AlvinoLionsTrombetti}, for example) to give a bound on the decreasing rearrangement of our subsolution.

On the other hand, the main drawback of the symmetrization technique is that it does not seem to work well when we combine it with cutoff functions; thus, we are not able to suitably modify it in order to directly show local estimates like \eqref{eq:up}. The idea to overcome this obstacle is to pass from small to large norms using a two-step procedure (in Section~\ref{secLocal}), utilizing the maximum principle. Thus, relying on Moser's estimate for the operator $\mathcal{L}_0u=-\dive(A\nabla u)+c\nabla u$ when the norm of $c$ is small, the first step is a perturbation argument based on the maximum principle that allows us to pass to the operator $\mathcal{L}$ when all the lower order terms have small norms. Then, the second step is an induction argument relying on the maximum principle (similar to the proofs of \cite[Propositions 3.4 and 7.8]{SakAPDE}), which allows us to pass to arbitrary norms for $b$ or $c$. To the best of our knowledge, the combination of the symmetrization technique with the fore mentioned argument in order to obtain local estimates has not appeared in the literature before (with the exception of \cite[Proposition 7.8]{SakAPDE}, which used estimates on Green's function), and it is one of the novelties of this article.

Since we do not obtain Moser's estimate \eqref{eq:up} using test functions and Moser's iteration, in order to deduce the reverse Moser estimate \eqref{eq:low} we transform supersolutions to subsolutions via exponentiation (in Section~\ref{secHarnack}). The advantage of this procedure is that, if the exponent is negative and close to $0$, we obtain a subsolution to an equation with the coefficients $b,d$ being small, thus we can apply \eqref{eq:up} to obtain a scale invariant estimate with ``good" constants, without any smallness assumptions (when $q<\infty$). This estimate has negative exponents appearing on the left hand side, and we show \eqref{eq:low} passing to positive exponents using an estimate for supersolutions and the John-Nirenberg inequality (as in \cite{MoserHarnack}). One drawback of this technique is that we do not obtain the full range $\alpha\in(0,\frac{n}{n-2})$ for the left hand side, as in \cite[Theorem 8.18]{Gilbarg}, but this does not affect the proof of the Harnack inequality. Then, the Harnack inequality and continuity of solutions are deduced combining \eqref{eq:up} and \eqref{eq:low}.

Finally, the optimality of our assumptions is shown in Section~\ref{secOptimality}. In particular, the sharpness of our spaces to guarantee some type of estimates (either having ``good" constants, or not) is shown, and the failure of scale invariant estimates with ``good" constants is exhibited by the construction in Proposition~\ref{dShouldBeSmall}.

\subsection*{Past works}

The first fundamental contribution to regularity for equations with rough coefficients was made by De Giorgi \cite{DeGiorgiContinuity} and Nash \cite{NashContinuity} and concerned H{\"o}lder continuity of solutions to the operator $-\dive(A\nabla u)=0$; a different proof, based on the Harnack inequality, was later given by Moser in \cite{MoserHarnack}. The literature concerning this subject is vast, and we refer to the books by Ladyzhenskaya and Ural'tseva \cite{LadyzhenskayaUraltseva} and Gilbarg and Trudinger \cite{Gilbarg}, as well as the references therein, for equations that also have lower order coefficients in $L^p$. However, in these results, the norms of those spaces are not scale invariant under the natural scaling of the equation, so it is not possible to obtain scale invariant estimates without extra assumptions on the coefficients (like smallness, for example). One instance of a scale invariant setting where $b,d,f,g\equiv 0$ and $c\in L^n$ was later treated by Nazarov and Ural'tseva in \cite{NazarovUraltsevaHarnack}.

Another well studied case of coefficients is the class of Kato spaces. The first work on estimates for Schr{\"o}dinger operators with the Laplacian and potentials in a suitable Kato class was by Aizenman and Simon in \cite{AizenmanSimonHarnack} using probabilistic techniques, which was later generalized (with nonprobabilistic techniques) by Chiarenza, Fabes and Garofalo in \cite{ChiarenzaFabesGarofalo}, allowing a second order part in divergence form. The case in \cite{AizenmanSimonHarnack} was also later treated using nonprobabilistic techniques by Simader \cite{SimaderElementary} and Hinz and Kalf \cite{HinzKalfSubsolutionEstimates}. In these works, $b,c\equiv 0$, while $d$ is assumed to belong to $K_n^{\loc}(\Omega)$, which is comprised of all functions $d$ in $\Omega$ such that $\eta_{\Omega_1,d}(r)\to 0$ as $r\to 0$, for all $\Omega_1$ compactly supported in $\Omega$, where
\[
\eta_{\Omega,d}(r)=\sup_{x\in\bR^n}\int_{\Omega\cap B_r(x)}\frac{|d(y)|}{|x-y|^{n-2}}\,dy
\]
(or, in some works, the supremum is considered over $x\in\Omega$). Moreover, adding the drift term $c\nabla u$, regularity estimates for $c$ in a suitable Kato class were shown by Kurata in \cite{KurataContinuity}.

From H{\"o}lder's inequality (see \eqref{eq:Holder}), if $d\in L^{\frac{n}{2},1}(\Omega)$, we have that
\[
\eta_{\Omega,d}(r)\leq C_n\sup_{x\in\bR^n}\|d\|_{L^{\frac{n}{2},1}(\Omega\cap B_r(x))}\xrightarrow[r\to 0]{} 0,
\]
therefore $L^{\frac{n}{2},1}(\Omega)\subseteq K_n^{\loc}(\Omega)$; that is, the class of Lorentz spaces we consider in this work is weaker than the Kato class. However, the constants in the results involving Kato classes depend on the rate of convergence of the function $\eta$ defined above to $0$, leading to different constants than the ones that we obtain in this article. More specifically, let $d\in L^{\frac{n}{2},1}(\bR^n)$ be supported in $B_1$, and set $d_M(x)=M^2d(Mx)$ for $M>0$. Then, we can show that $\eta_{B_1,d_M}(Mr)=\eta_{B_1,d}(r)$, thus the functions $\eta_{B_1,d_M},\eta_{B_1,d}$ do not converge to $0$ at the same rate. Hence, the estimates shown using techniques involving Kato spaces, and concerning subsolutions $u_M$ to
\[
-\Delta u_M+d_Mu_M\leq 0
\]
in $B_1$, lead to constants that could blow up as $M\to\infty$. On the other hand, the $L^{\frac{n}{2},1}(B_1)$ norm of $d_M$ is bounded above uniformly in $M$, hence the results we prove in this article are not direct consequences of their counterparts involving Kato classes.

Finally, considering all the lower order terms, Mourgoglou in \cite{MourgoglouRegularity} shows regularity estimates when the coefficients $b,d$ belong to the scale invariant Dini type Kato-Stummel classes (see \cite[Section 2.2]{MourgoglouRegularity}), and also constructs Green's functions. However, the framework we consider in this article for the Moser estimate and Harnack's inequality, as well as our techniques, are different from the ones in \cite{MourgoglouRegularity}. For example, focusing on the case when $c,d\equiv 0$, the coefficient $b$ in \cite[Theorems 4.4, 4.5 and 4.12]{MourgoglouRegularity} is assumed to be such that $|b|^2\in\mathcal{K}_{{\rm Dini},2}$, which does not cover the case $b\in L^{n,1}$, since for any $\alpha>1$, the function $b(x)=x|x|^{-2}\left(-\ln|x|\right)^{-a}$ is a member of $L^{n,1}(B_{1/e})$, while $|b|^2\notin\mathcal{K}_{{\rm Dini},2}(B_{1/e})$.

We conclude with a brief discussion on symmetrization techniques. Such a technique was used by Weinberger in \cite{WeinbergerSymmetrization} in order to show boundedness of solutions with vanishing trace to $-\dive(A\nabla u)=-\dive f$ and $-\dive(A\nabla u)=g$, where $f\in L^p$ and $g\in L^{\frac{p}{2}}$, $p>n$. Another well known technique consists of a use of test functions that leads to bounds for the derivative of the integral of $|\nabla u|^2$ over superlevel sets of $u$, where $u$ is a subsolution to $\mathcal{L}u\leq-\dive f+g$. This bound, combined with Talenti's inequality \cite[estimate (40)]{TalentiElliptic}, gives an estimate for the derivative of the decreasing rearrangement of $u$, leading to bounds for $u$ in various spaces and comparison results. This technique has been used by many authors in order to study regularity properties of solutions to second order pdes, some works being \cite{AlvinoTrombetti78}, \cite{AlvinoTrombetti81}, \cite{AlvinoLionsTrombetti}, \cite{BettaMercaldoComparisonAndRegularity}, \cite{DelVecchioPosteraroExistenceMeasure}, \cite{DelVecchioPosteraroNoncoercive}, \cite{AlvinoTrombettiLionsMatarasso}, \cite{AlvinoFeroneTrombetti}, \cite{Buccheri}. However, as we mentioned above, to the best of our knowledge, no local boundedness results have been deduced using this method so far.

We also mention that, in order to treat lower order coefficients, pseudo-rearrangements of functions are also considered the literature, which are derivatives of integrals over suitable sets $\Omega(s)\subseteq\Omega$ (see, for example, \cite[page 11]{SakAPDE}). On the contrary, in this work we avoid this procedure, and as we mentioned above we rely instead on a slightly different approach, inspired by \cite{CianchiMazyaSchrodinger}.

\vspace{4mm}

\begin{acknowledgments*}
We would like to thank Professors Carlos Kenig and Andrea Cianchi for useful conversations regarding some parts of this article.
\end{acknowledgments*}

\section{Preliminaries}

\subsection{Definitions}

If $\Omega\subseteq\bR^n$ is a domain, $W_0^{1,2}(\Omega)$ will be the closure of $C_c^{\infty}(\Omega)$ under the $W^{1,2}$ norm, where
\[
\|u\|_{W^{1,2}(\Omega)}=\|u\|_{L^2(\Omega)}+\|\nabla u\|_{L^2(\Omega)}.
\]
When $\Omega$ has infinite measure, the space $W^{1,2}(\Omega)$ is not well suited to the problems we consider. For this reason, we let $Y_0^{1,2}(\Omega)$ be the closure of $C_c^{\infty}(\Omega)$ under the $Y^{1,2}$ norm, where
\begin{equation}\label{eq:Y}
\|u\|_{Y^{1,2}(\Omega)}=\|u\|_{L^{2^*}(\Omega)}+\|\nabla u\|_{L^2(\Omega)},
\end{equation}
and $2^*=\frac{2n}{n-2}$ is the Sobolev conjugate to $2$. From the Sobolev inequality
\[
\|\phi\|_{L^{2^*}(\Omega)}\leq C_n\|\nabla\phi\|_{L^2(\Omega)},
\]
for all $\phi\in C_c^{\infty}(\Omega)$, we have that $Y_0^{1,2}(\Omega)=W_0^{1,2}(\Omega)$ in the case $|\Omega|<\infty$. We also set $Y^{1,2}(\Omega)$ to be the space of weakly differentiable $u\in L^{2^*}(\Omega)$, such that $\nabla u\in L^2(\Omega)$, with the $Y^{1,2}$ norm. 

If $u$ is a measurable function in $\Omega$, we define the distribution function
\begin{equation}\label{eq:distrFun}
\mu_u(t)=\left|\left\{x\in\Omega: |u(x)|>t\right\}\right|,\quad t>0.
\end{equation}
If $u\in L^p(\Omega)$ for some $p\geq 1$, then $\mu_u(t)<\infty$ for any $t>0$. Moreover, we define the decreasing rearrangement of $u$ by
\begin{equation}\label{eq:DecrRearr}
u^*(\tau)=\inf\{t>0:\mu_u(t)\leq \tau\},
\end{equation}
as in \cite[(1.4.2), page 45]{Grafakos}. Then, $u^*$ is equimeasurable to $u$: that is,
\begin{equation}\label{eq:equimeasurable}
\left|\left\{x\in\Omega:|u(x)|>t\right\}\right|=\left|\left\{s>0:u^*(s)>t\right\}\right|\,\,\,\text{for all}\,\,\,t>0.
\end{equation}
Given a function $f\in L^p(\Omega)$, we consider its maximal function
\begin{equation}\label{eq:MaximalFunction}
\mathcal{M}_{f}(\tau)=\frac{1}{\tau}\int_0^{\tau}f^*(\sigma)\,d\sigma,\quad \tau>0.
\end{equation}
Let $p\in(0,\infty)$ and $q\in(0,\infty]$. If $f$ is a function defined in $\Omega$, we define the Lorentz seminorm
\begin{equation}\label{eq:LorentzDfn}
\|f\|_{L^{p,q}(\Omega)}=\left\{\begin{array}{c l} \displaystyle \left(\int_0^{\infty}\left(\tau^{\frac{1}{p}}f^*(\tau)\right)^q\frac{d\tau}{\tau}\right)^{\frac{1}{q}}, & q<\infty \\ \displaystyle\sup_{\tau>0}\tau^{\frac{1}{p}}f^*(\tau),& q=\infty,\end{array}\right.
\end{equation}
as in \cite[Definition 1.4.6]{Grafakos}. We say that $f\in L^{p,q}(\Omega)$ if $\|f\|_{L^{p,q}(\Omega)}<\infty$. Then $\|\cdot\|_{p,q}$ is indeed a seminorm, since
\begin{equation}\label{eq:Seminorm}
\|f+g\|_{p,q}\leq C_{p,q}\|f\|_{p,q}+C_{p,q}\|g\|_{p,q},
\end{equation}
from \cite[(1.4.9), page 50]{Grafakos}. In addition, from \cite[Proposition 1.4.10]{Grafakos}, Lorentz spaces increase if we increase the second index, with
\begin{equation}\label{eq:LorentzNormsRelations}
\|f\|_{L^{p,r}}\leq C_{p,q,r}\|f\|_{L^{p,q}}\,\,\,\,\text{for all}\,\,\,0<p<\infty,\,\,0<q<r\leq\infty.
\end{equation}
H{\"o}lder's inequality for Lorentz functions states that
\begin{equation}\label{eq:Holder}
\|fg\|_{L^{p,q}}\leq C_{p_1,q_1,p_2,q_2}\|f\|_{L^{p_1,q_1}}\|g\|_{L^{p_2,q_2}},
\end{equation}
whenever $0<p,p_1,p_2<\infty$ and $0<q,q_1,q_2\leq\infty$ satisfy the relations $\frac{1}{p}=\frac{1}{p_1}+\frac{1}{p_2}$, $\frac{1}{q}=\frac{1}{q_1}+\frac{1}{q_2}$ (see \cite[Exercise 1.4.19]{Grafakos}).

If $p\in(1,\infty]$ and $q\in[1,\infty)$, then \cite[Theorem 3.21, page 204]{SteinWeiss} implies that
\begin{equation}\label{eq:maximalInequality}
\|\mathcal{M}_f\|_{p,q}\leq C_p\|f\|_{p,q},
\end{equation}
where $\mathcal{M}_f$ is the maximal function defined in \eqref{eq:MaximalFunction}.

For a function $u\in Y^{1,2}$, we will say that $u\leq s$ on $\partial\Omega$ if $(u-s)^+=\max\{u-s,0\}\in Y_0^{1,2}(\Omega)$. Moreover, $\sup_{\partial\Omega}u$ will be defined as the infimum of all $s\in\mathbb R$ such that $u\leq s$ on $\partial\Omega$.

We now turn to the definitions of subsolutions, supersolutions and solutions. For this, let $\Omega\subseteq\bR^n$ be a domain, and let $A$ be bounded in $\Omega$, $b,c\in L^{n,\infty}(\Omega)$, $d\in L^{\frac{n}{2},\infty}(\Omega)$ and $f,g\in L^1_{\loc}(\Omega)$. If $\mathcal{L}u=-\dive(A\nabla u+bu)+c\nabla u+du$, we say that $u\in Y^{1,2}(\Omega)$ is a solution to the equation $\mathcal{L}u=-\dive f+g$ in $\Omega$, if
\[
\int_{\Omega}A\nabla u\nabla\phi+b\nabla\phi\cdot u+c\nabla u\cdot\phi+du\phi=\int_{\Omega}f\nabla\phi+g\phi,\,\,\,\text{for all}\,\,\,\phi\in C_c^{\infty}(\Omega).
\]
Moreover, we say that $u\in Y^{1,2}(\Omega)$ is a subsolution to $\mathcal{L}u\leq-\dive f+g$ in $\Omega$, if
\begin{equation}\label{eq:subsolDfn}
\int_{\Omega}A\nabla u\nabla\phi+b\nabla\phi\cdot u+c\nabla u\cdot\phi+du\phi\leq\int_{\Omega}f\nabla\phi+g\phi,\,\,\,\text{for all}\,\,\,\phi\in C_c^{\infty}(\Omega),\,\phi\geq 0.
\end{equation}
We also say that $u$ is a supersolution to $\mathcal{L}u\geq-\dive f+g$, if $-u$ is a subsolution to $\mathcal{L}(-u)\leq \dive f-g$.

\subsection{Main lemmas}

We now discuss some lemmas that we will use in the sequel. We begin with the following estimate, in which we show that a function in $L^{n,q}$ for $q>1$ fails to be in $L^{n,1}$ by a logarithm, with constant as small as we want. This fact will be useful in the proof of Lemma~\ref{MainEstimate}.

\begin{lemma}\label{NormWithE}
Let $f\in L^{n,q}(\Omega)$ for some $q\in(1,\infty)$. Then, for any $0<\sigma_1<\sigma_2<\infty$ and $\e>0$,
\[
\int_{\sigma_1}^{\sigma_2}\tau^{\frac{1}{n}-1}f^*(\tau)\,d\tau\leq \e\ln\frac{\sigma_2}{\sigma_1}+C\|f\|_{n,q}^q,
\]
where $C$ depends on $q$ and $\e$.
\end{lemma}
\begin{proof}
Let $p\in(1,\infty)$ be the conjugate exponent to $q$. Then, from H{\"o}lder's inequality and \eqref{eq:LorentzDfn},
\begin{align*}
\int_{\sigma_1}^{\sigma_2}\tau^{\frac{1}{n}-1}f^*(\tau)\,d\tau&=\int_{\sigma_1}^{\sigma_2}\tau^{-\frac{1}{p}}\tau^{\frac{1}{n}-\frac{1}{q}}f^*(\tau)\,d\tau\leq \left(\int_{\sigma_1}^{\sigma_2}\tau^{-1}\,d\tau\right)^{\frac{1}{p}}\left(\int_{\sigma_1}^{\sigma_2}\tau^{\frac{q}{n}-1}f^*(\tau)^q\,d\tau\right)^{\frac{1}{q}}\\
&\leq\left(p\e\ln\frac{\sigma_2}{\sigma_1}\right)^{\frac{1}{p}}\cdot(p\e)^{-\frac{1}{p}}\|f\|_{n,q}\leq\e\ln\frac{\sigma_2}{\sigma_1}+\frac{(p\e)^{-\frac{q}{p}}}{q}\|f\|_{n,q}^q,
\end{align*}
where we also used Young's inequality for the last step.
\end{proof}

The following describes the behavior of the Lorentz seminorm on disjoint sets.

\begin{lemma}\label{NormDisjoint}
Let $\Omega\subseteq\bR^n$ be a set, and let $X,Y$ be nonempty and disjoint subsets of $\Omega$. If $f\in L^{p,q}(\Omega)$ for some $p,q\in[1,\infty)$, then
\[
\|f\|_{L^{p,q}(\Omega)}^r\geq\|f\|_{L^{p,q}(X)}^r+\|f\|_{L^{p,q}(Y)}^r,\quad r=\max\{p,q\}.
\]
\end{lemma}
\begin{proof}
Let $\mu$, $\mu_X$, $\mu_Y$ be the distribution functions of $f,f|_X$ and $f|_Y$, respectively. As in \cite[Lemma 2.4]{SakAPDE}, we have that $\mu\geq\mu_X+\mu_Y$. Also, if $p\geq q$, then $\frac{q}{p}\leq 1$, hence the reverse Minkowski inequality shows that
\[
\left(\int_0^{\infty}(\mu_X(t)+\mu_Y(t))^{\frac{q}{p}}s^{q-1}\,ds\right)^{\frac{p}{q}}\geq\left(\int_0^{\infty}\mu_X(t)^{\frac{q}{p}}s^{q-1}\,ds\right)^{\frac{p}{q}}+\left(\int_0^{\infty}\mu_Y(t)^{\frac{q}{p}}s^{q-1}\,ds\right)^{\frac{p}{q}}.
\]
On the other hand, if $q>p$, then $\frac{q}{p}>1$, hence $a^{\frac{q}{p}}+b^{\frac{q}{p}}\leq (a+b)^{\frac{q}{p}}$ for all $a,b>0$. Therefore,
\[
\int_0^{\infty}(\mu_X(t)+\mu_Y(t))^{\frac{q}{p}}s^{q-1}\,ds\geq\int_0^{\infty}\mu_X(t)^{\frac{q}{p}}s^{q-1}\,ds+\int_0^{\infty}\mu_Y(t)^{\frac{q}{p}}s^{q-1}\,ds.
\]
Then, the proof follows from the expression for the $L^{p,q}$ seminorm in \cite[Proposition 1.4.9]{Grafakos}.
\end{proof}

The next lemma will be useful in order to reduce to the case $d=0$.

\begin{lemma}\label{Reduction}
Let $\Omega\subseteq\bR^n$ be a domain, and $d\in L^{\frac{n}{2},1}(\Omega)$. Then there exists a weakly differentiable vector valued function $e\in L^{n,1}(\Omega)$, with $\dive e=d$ in $\Omega$ and $\|e\|_{L^{n,1}(\Omega)}\leq C_n\|d\|_{L^{\frac{n}{2},1}(\Omega)}$.
\end{lemma}
\begin{proof}	Extend $d$ by $0$ outside $\Omega$, and consider the Newtonian potential $v$ of $d$; that is, we set
\[
w(x)=C_n\int_{\bR^n}\frac{d(y)}{|x-y|^{n-2}}\,dy.
\]
From \cite[Theorem 9.9]{Gilbarg} we have that $w$ is twice weakly differentiable in $\Omega$, and $\Delta w=d$. Setting $e=\nabla w$, we have that $\dive e=d$. Moreover, $|e(x)|=|\nabla w(x)|\leq C_n\int_{\bR^n}\frac{|d(y)|}{|x-y|^{n-1}}\,dy,$
and the estimate follows from the first part of \cite[Exercise 1.4.19]{Grafakos}.
\end{proof}

The next lemma shows that $u^*$ is locally absolutely continuous, when $u\in Y^{1,2}$.

\begin{lemma}\label{uAbsCts}
Let $\Omega$ be a domain and $u\in Y_0^{1,2}(\Omega)$. Then $u^*$ is absolutely continuous in $(a,b)$, for any $0<a<b<\infty$.
\end{lemma}
\begin{proof}
Extending $u$ by $0$ outside $\Omega$, we may assume that $u\in Y^{1,2}(\bR^n)$. 
	
Consider the function $u^*$ defined in \cite[(2), page 153]{BrothersZiemer} (this $u^*$ is not the same as the one in \eqref{eq:DecrRearr}!), and the function $\tilde{u}(|x|)=u^*(x)$ (as in \cite[page 154]{BrothersZiemer}). Then, from the argument for the proof of \cite[Lemma 2.6]{SakAPDE}, it is enough to show that $\tilde{u}$ is locally absolutely continuous in $(0,\infty)$.
	
To show this, note that the proof of \cite[Lemma 2.4]{BrothersZiemer} shows that $u^*\in Y^{1,2}(\bR^n)$ whenever $u\in Y^{1,2}(\bR^n)$ (since $Y^{1,2}(\bR^n)$ is reflexive, bounded sequences have subsequences that converge weakly, and the rest of the argument runs unchanged). Hence, $u^*\in W^{1,2}_{\loc}(\bR^n)$, and combining with \cite[Proposition 2.5]{BrothersZiemer}, we obtain that $\tilde{u}$ is locally absolutely continuous in $(0,\infty)$, as in \cite[Corollary 2.6]{BrothersZiemer}, which completes the proof.
\end{proof}

We now turn to the following decomposition, which in similar to \cite[Lemma 2.8]{SakAPDE}. This will be useful in a change of variables that we will perform in Lemma~\ref{Talenti}, as well as in the proof of the estimate in Lemma~\ref{MainEstimate}.

\begin{lemma}\label{Splitting}
Let $\Omega\subseteq\bR^n$ be a domain, and let $u\in Y_0^{1,2}(\Omega)$. Then we can write
\[
(0,\infty)=G_u\cup D_u\cup N_u,
\]
where the union is disjoint, such that the following hold.
\begin{enumerate}[i)]
\item If $x\in G_u$, then $u^*$ is differentiable at $x$, $\mu_u$ is differentiable at $u^*(x)$, and $(u^*)'(x)\neq 0$. Moreover,
\begin{equation}\label{eq:xGuFormulas}
\mu_u(u^*(x))=x\quad\text{and}\quad\mu_u'(u^*(x))=\frac{1}{u^*(x)},\quad\text{for all}\quad x\in G_u.
\end{equation}
\item If $x\in D_u$, then $u^*$ is differentiable at $x$, with $(u^*)'(x)=0$.
\item $N_u$ is a null set.
\end{enumerate}
\end{lemma}
\begin{proof}
The proof is the same as the proof of \cite[Lemma 2.8]{SakAPDE}, where we use continuity of $u^*$ shown in Lemma~\ref{uAbsCts}, instead of \cite[Lemma 2.6]{SakAPDE}.
\end{proof}

We now turn to the following lemma, which is based on \cite[Lemma 3.1]{CianchiMazyaSchrodinger}. As we mentioned in the introduction, the properties of the function $\Psi$ defined below will be crucial in the proof of Lemma~\ref{MainEstimate} and, using this lemma, we avoid the construction of pseudo-rearrangements (as in \cite[pages 11 and 12]{SakAPDE}.

\begin{lemma}\label{Talenti}
Let $\Omega\subseteq\bR^n$ be a domain and $u\in Y_0^{1,2}(\Omega)$ with $u\geq 0$. For any $f\in L^1(\Omega)$, the function
\[
R_{f,u}(\tau)=\int_{[u>u^*(\tau)]}|f|
\]
is absolutely continuous in $(0,\infty)$, and if $\Psi_{f,u}=R'_{f,u}\geq 0$ is its derivative, then for any $p>1$ and $q\geq 1$,
\begin{equation}\label{eq:PsiEst}
\|\Psi_{f,u}\|_{L^{p,q}(0,\infty)}\leq C_{p,q}\|f\|_{L^{p,q}(\Omega)}.
\end{equation}
Moreover, for almost every $\tau>0$,
\begin{equation}\label{eq:Talenti}
(-u^*)'(\tau)\leq C_n\tau^{\frac{1}{n}-1}\sqrt{\Psi_{|\nabla u|^2,u}(\tau)}.
\end{equation}
\end{lemma}
\begin{proof}
Let $u^{\circ}$ be the function defined in \cite[page 660]{CianchiMazyaSchrodinger}; that is, we define
\[
u^{\circ}(\tau)=\sup\{t':\mu_u(t')\geq \tau\},
\]
where $\mu_u$ coincides with our definition of the distribution function \eqref{eq:distrFun}, since $u\geq 0$. We will show that $u^*=u^{\circ}$, so that $R_{f,u}$ coincides with the function in \cite[Lemma 3.1]{CianchiMazyaSchrodinger}. Then, the proof of the same lemma (where for absolute continuity of $u^*$, we will use Lemma~\ref{uAbsCts}) will show absolute continuity of $R_{f,u}$, and \eqref{eq:PsiEst} will follow from \cite[(3.12), page 661]{CianchiMazyaSchrodinger} and \eqref{eq:maximalInequality} .

Note first that, from the definitions, $u^*(\tau)\leq u^{\circ}(\tau)$ for all $\tau$. If now $u^*(\tau)<u^{\circ}(\tau)$, then we can find $t<t'$ with $\mu_u(t)\leq\tau$ and $\mu_u(t')\geq \tau$. Since $\mu_u$ is decreasing, this will imply that $\mu_u(t')\leq \mu_u(t)$, hence $\mu_u$ is equal to $\tau$ in $[t,t']$, which is a contradiction with continuity of $u^*$ from Lemma~\ref{uAbsCts}. This shows that $u^{\circ}=u^*$, and completes the proof of the first part.
	
To show estimate \eqref{eq:Talenti}, set $T_u(t)=\displaystyle\int_{[u>t]}|\nabla u|^2$, and note that, from \cite[estimate (40)]{TalentiElliptic},
\begin{equation}\label{eq:FromTalenti}
C_n\leq \mu_u(t)^{\frac{2}{n}-2}(-\mu_u'(t))\left(-\frac{d}{dt}\int_{[u>t]}|\nabla u|^2\right)=C_n\mu_u(t)^{\frac{2}{n}-2}(-\mu_u'(t))(-T_u)'(t),
\end{equation}
for every $t\in F$, where $F\subseteq(0,\sup_{\Omega}u)$ has full measure (this estimate is shown for $u\in W_0^{1,2}(\Omega)$, but the same proof as in \cite[pages 711-712]{TalentiElliptic} gives the result for $u\in Y_0^{1,2}(\Omega)$).

Consider now the splitting $(0,\infty)=G_u\cup D_u\cup N_u$ in Lemma~\ref{Splitting}. We claim that $u^*(\tau)\in F$ for almost every $\tau\in G_u$: if this is not the case, then there exists $G\subseteq G_u$, with positive measure, such that if $\tau\in G$, then $u^*(\tau)\notin F$. Then, the set $u^*(G)$ has measure zero and $u^*$ is differentiable at every point $\tau\in G$, hence \cite[Theorem 1]{SerrinVarberg} shows that $(u^*)'(\tau)=0$ for almost every $\tau\in G$. However, $u^*(\tau)\neq 0$ for every $\tau\in G_u$ from Lemma~\ref{Splitting}, which is a contradiction with the fact that $G$ has positive measure. So, $u^*(\tau)\in F$ for almost every $\tau\in G_u$, and for those $\tau$, plugging $u^*(\tau)$ in \eqref{eq:FromTalenti}, we obtain that
\[
C_n\leq \mu_u(u^*(\tau))^{\frac{2}{n}-2}(-\mu_u'(u^*(\tau))(-T_u)'(u^*(\tau)),
\]
and using \eqref{eq:xGuFormulas}, we obtain that
\[
(-u^*)'(\tau)\leq C_n\tau^{\frac{2}{n}-2}(-T_u)'(u^*(\tau)),
\]
for almost every $\tau\in G$. Moreover, $R_{|\nabla u|^2,u}=T_u\circ u^*$, and since $T_u$ is differentiable at $u^*(\tau)$ for almost every $\tau\in G_u$, multiplying the last estimate with $(-u^*)'(\tau)$ implies that
\[
\left((-u^*)'(\tau)\right)^2\leq C_n\tau^{\frac{2}{n}-2}(-T_u)'(u^*(\tau))\cdot (-u^*)'(\tau)=C_n\tau^{\frac{2}{n}-2}R_{|\nabla u|^2,u}'(\tau),
\]
which shows that \eqref{eq:Talenti} holds for almost every $\tau\in G_u$. On the other hand, $(u^*)'(\tau)=0$ when $\tau\in D_u$, so \eqref{eq:Talenti} also holds for almost every $\tau\in D_u$. Since $N_u$ has measure zero, \eqref{eq:Talenti} holds almost everywhere in $(0,\infty)$, which completes the proof.
\end{proof}

Finally, the following is a Gr{\"o}nwall type lemma, which we prove in the setting that will appear in Lemma~\ref{MainEstimate}. The reason for this is that the function $g_2g_3$ will not necessarily be integrable close to $0$, which turns out to be inconsequential.

\begin{lemma}\label{Gronwall}
Let $M>0$, and suppose that $f,g_1,g_2,g_3$ are functions defined in $(0,M)$, with $g_2,g_3\geq 0$. Assume that $g_2g_3$ is locally integrable in $(0,M)$, $g_3f\in L^1(0,M)$ and
\[
\exp\left(-\int_{\tau_0}^{\tau}g_2g_3\right)g_1(\tau)g_3(\tau)\in L^1(0,M),\qquad \exp\left(\int_{\e}^{\tau_0}g_2g_3\right)\int_0^{\e}g_3f\xrightarrow[\e\to 0]{}0,
\]
for some $\tau_0\in(0,M)$. If $\displaystyle f(\tau)\leq g_1(\tau)+g_2(\tau)\int_0^{\tau}g_3f$ in $(0,M)$, then, for every $\tau\in(0,M)$,
\[
f(\tau)\leq g_1(\tau)+g_2(\tau)\int_0^{\tau}g_1(\sigma)g_3(\sigma)\exp\left(\int_{\sigma}^{\tau}g_2(\rho)g_3(\rho)\,d\rho\right)\,d\sigma.
\]
\end{lemma}
\begin{proof}
Define $G(\tau)=\int_0^{\tau}g_3f$ and $H(\tau)=\int_{\tau_0}^{\tau}g_2g_3$, then $G$ is absolutely continuous in $[0,M]$ and $H$ is locally absolutely continuous in $(0,M)$. Then, we have that $(e^{-H}G)'=e^{-H}(G'-H'G)\leq e^{-H}g_1g_3$, and since $e^{-H}G$ is absolutely continuous in $(\e,\tau)$ for $0<\e<\tau<M$, we integrate to obtain
\[
e^{-H(\tau)}G(\tau)-e^{-H(\e)}G(\e)\leq\int_{\e}^{\tau}e^{-H}g_1g_3.
\]
The proof is complete after letting $\e\to 0$ and plugging the last estimate in the original estimate for $f$.
\end{proof}

\section{Global estimates}\label{secGlobal}

\subsection{The main estimate}

The following lemma is the main estimate that will lead to global boundedness for subsolutions. The test function we use comes from \cite[page 663, proof of Theorem 2.1]{CianchiMazyaSchrodinger} and it is a slight modification of test functions that have been used in the literature before (see, for example, the references for the decreasing rearrangements technique in the introduction).

\begin{lemma}\label{MainEstimate}
Let $\Omega\subseteq\bR^n$ be a domain. Let $A$ be uniformly elliptic and bounded in $\Omega$, with ellipticity $\lambda$. Let also $b,f\in L^{n,1}(\Omega)$ and $g\in L^{\frac{n}{2},1}(\Omega)$. There exists $\nu=\nu_{n,\lambda}$ such that, if $c=c_1+c_2\in L^{n,\infty}(\Omega)$ with $c_1\in L^{n,q}(\Omega)$ for some $q<\infty$ and $\|c_2\|_{n,\infty}<\nu$, then for any subsolution $u\in Y_0^{1,2}(\Omega)$ to
\[
-\dive(A\nabla u+bu)+c\nabla u\leq -\dive f+g
\]
in $\Omega$, and any $\tau\in(0,1)$,
\begin{equation}
\begin{split}\label{eq:Main}
-v'(\tau)\leq C_1\tau^{\frac{1}{n}-1}\sqrt{\Psi_{|f|^2}(\tau)}+C_1\tau^{\frac{2}{n}-1}\mathcal{M}_g(\tau)+C_1e^{C_2\|c_1\|_{n,q}^q}\tau^{\frac{1}{n}-\frac{3}{2}}\int_0^{\tau}\sigma^{\frac{1}{n}-\frac{1}{2}}\sqrt{\Psi_{|f|^2}(\sigma)\Psi_{|c|^2}(\sigma)}\,d\sigma\\
+C_1e^{C_2\|c_1\|_{n,q}^q}\tau^{\frac{1}{n}-\frac{3}{2}}\int_0^{\tau}\sigma^{\frac{2}{n}-\frac{1}{2}}\mathcal{M}_{g}(\sigma)\sqrt{\Psi_{|c|^2}(\sigma)}\,d\sigma\\
+C_1v(\tau)\tau^{\frac{1}{n}-1}\sqrt{\Psi_{|b|^2}(\tau)}+C_1e^{C_2\|c_1\|_{n,q}^q}\tau^{\frac{1}{n}-\frac{3}{2}}\int_0^{\tau}\sigma^{\frac{1}{n}-\frac{1}{2}}v(\sigma)\sqrt{\Psi_{|b|^2}(\sigma)\Psi_{|c|^2}(\sigma)}\sigma^{\frac{1}{n}-\frac{1}{2}}\,d\sigma,
\end{split}
\end{equation}
where $C_1$ depends on $n,\lambda$, $C_2$ depends on $n,\lambda,q$, and where $v=(u^+)^*$ is the decreasing rearrangement of $u^+$, $\mathcal{M}_g$ is as in \eqref{eq:MaximalFunction}, and $\Psi_{|b|^2}=\Psi_{|b|^2,u^+},\Psi_{|c|^2}=\Psi_{|c|^2,u^+}, \Psi_{|f|^2}=\Psi_{|f|^2,u^+}$ are defined in Lemma~\ref{Talenti}.
\end{lemma}
\begin{proof}
Fix $\tau,h>0$, and consider the test function
\[
\psi=\left\{\begin{array}{l l} 0, & 0\leq u^+\leq v(\tau+h) \\ u-v(\tau+h), & v(\tau+h)< u^+\leq v(\tau) \\ v(\tau)-v(\tau+h) & u^+>v(\tau).\end{array}\right.
\]
Since $\psi\in W_0^{1,2}(\Omega)$ and $\psi\geq 0$, we can use it as a test function, and from ellipticity of $A$,
\begin{multline*}
\lambda\int_{[v(\tau+h)<u\leq v(\tau)]}|\nabla u|^2\leq v(\tau)\int_{[v(\tau+h)<u\leq v(\tau)]}|b\nabla u|+(v(\tau)-v(\tau+h))\int_{[u>v(\tau+h)]}|c\nabla u|\\
+\int_{[v(\tau+h)<u\leq v(\tau)]}|f\nabla u|+(v(\tau)-v(\tau+h))\int_{[u>v(\tau+h)]}|g|.
\end{multline*}
Letting $\Psi(\tau)=\Psi_{|\nabla u|^2,u^+}$ (as in Lemma~\ref{Talenti}), dividing by $h$, using the Cauchy-Schwartz inequality and letting $h\to 0$, we obtain that
\begin{multline}\label{eq:p0}
\Psi(\tau)\leq C_{\lambda}v(\tau)\sqrt{\Psi_{|b|^2}(\tau)}\sqrt{\Psi(\tau)}+C_{\lambda}(-v')(\tau)\int_{[u>v(\tau)]}|c\nabla u|\\
+C_{\lambda}\sqrt{\Psi_{|f|^2}(\tau)}\sqrt{\Psi(\tau)}+C_{\lambda}(-v'(\tau))\int_{[u>v(\tau)]}|g|,
\end{multline}
where we also used continuity of the functions $R_{|c\nabla u|,u^+}$ and $R_{|g|,u^+}$, from Lemma~\ref{Talenti}. Moreover, from absolute continuity of $R_{|c\nabla u|, u^+}$ and the Cauchy-Schwartz inequality, we obtain
\begin{equation}\label{eq:p1}
\int_{[u>v(\tau)]}|c\nabla u|=\int_0^{\tau}\Psi_{|c\nabla u|,u^+}\leq\int_0^{\tau}\sqrt{\Psi_{|c|^2}}\sqrt{\Psi}.
\end{equation}
Let now $\mu$ be the distribution function of $u^+$, and consider the decomposition $(0,\infty)=G_{u^+}\cup D_{u^+}\cup N_{u^+}$ from Lemma~\ref{Splitting}. Then, for $\tau\in G_{u^+}$, the Hardy-Littlewood inequality (see, for example, \cite[page 44, Theorem 2.2]{BennettSharpley}) and \eqref{eq:xGuFormulas} show that
\[
(-v'(\tau))\int_{[u>v(\tau)]}|g|\leq (-v'(\tau))\int_0^{\mu(v(\tau))}g^*=(-v'(\tau))\int_0^{\tau}g^*=\tau(-v'(\tau))\mathcal{M}_g(\tau).
\]
On the other hand, if $\tau\in N_{u^+}$, then $-v'(\tau)=0$, and since $N_{u^+}$ has measure $0$, the last estimate holds almost everywhere. Hence, plugging the last estimate and \eqref{eq:p1} in \eqref{eq:p0}, we obtain that
\begin{multline*}
\Psi(\tau)\leq C_{\lambda}v(\tau)\sqrt{\Psi_{|b|^2}(\tau)}\sqrt{\Psi(\tau)}+C_{\lambda}(-v')(\tau)\int_0^{\tau}\sqrt{\Psi_{|c|^2}}\sqrt{\Psi}\\
+C_{\lambda}\sqrt{\Psi_{|f|^2}(\tau)}\sqrt{\Psi(\tau)}+C_{\lambda}\tau(-v'(\tau))\mathcal{M}_g(\tau).
\end{multline*}
Let $\tau$ such that $\Psi(\tau)>0$. Then, dividing the last estimate by $\sqrt{\Psi(\tau)}$ and using \eqref{eq:Talenti},
\begin{align*}
\sqrt{\Psi(\tau)}&\leq C_{\lambda}v(\tau)\sqrt{\Psi_{|b|^2}(\tau)}+\frac{C_{\lambda}(-v')(\tau)}{\sqrt{\Psi(\tau)}}\int_0^{\tau}\sqrt{\Psi_{|c|^2}}\sqrt{\Psi}+C_{\lambda}\sqrt{\Psi_{|f|^2}(\tau)}+\frac{C_{\lambda}\tau(-v'(\tau))\mathcal{M}_g(\tau)}{\sqrt{\Psi(\tau)}}\\
&\leq C\sqrt{\Psi_{|f|^2}(\tau)}+C\tau^{\frac{1}{n}}\mathcal{M}_g(\tau)+Cv(\tau)\sqrt{\Psi_{|b|^2}(\tau)}+C\tau^{\frac{1}{n}-1}\int_0^{\tau}\sqrt{\Psi_{|c|^2}}\sqrt{\Psi},
\end{align*}
where $C=C_{n,\lambda}$. On the other hand, the last estimate holds also when $\Psi(\tau)=0$, hence it holds for almost every $\tau>0$.

Note now that, from subadditivity of $\Psi$ and Lemma~\ref{NormWithE} (since we can assume that $q>1$) for any $\e>0$,
\begin{align*}
\int_{\sigma}^{\tau}\rho^{\frac{1}{n}-1}\sqrt{\Psi_{|c|^2}(\rho)}\,d\rho&\leq \int_{\sigma}^{\tau}\rho^{\frac{1}{n}-1}\sqrt{\Psi_{|c_1|^2}(\rho)}\,d\rho+\int_{\sigma}^{\tau}\rho^{\frac{1}{n}-1}\sqrt{\Psi_{|c_2|^2}(\rho)}\,d\rho\\
&\leq \e\ln\frac{\tau}{\sigma}+C_{q,\e}\left\|\sqrt{\Psi_{|c_1|^2}}\right\|_{n,q}^q+\int_{\sigma}^{\tau}\rho^{\frac{1}{n}-1}\left\|\sqrt{\Psi_{|c_2|^2}}\right\|_{n,\infty}\rho^{-\frac{1}{n}}\,d\rho\\
&\leq \e\ln\frac{\tau}{\sigma}+C_{n,q,\e}\|c_1\|_{n,q}^q+C_n\|c_2\|_{n,\infty}\ln\frac{\tau}{\sigma}.
\end{align*}
We choose $\e=\e_{n,\lambda}$ and $\nu_{n,\lambda}$ such that $C\e_{n,\lambda}+CC_n\nu_{n,\lambda}\leq\frac{1}{2}-\frac{1}{n}$; then, we will have that
\begin{equation}\label{eq:g2g3}
\exp\left(C\int_{\sigma}^{\tau}\rho^{\frac{1}{n}-1}\sqrt{\Psi_{|c|^2}(\rho)}\,d\rho\right)\leq e^{C_2\|c_1\|_{n,q}^q}\left(\frac{\tau}{\sigma}\right)^{\frac{1}{2}-\frac{1}{n}},
\end{equation}
where $C_2$ depends on $n,q$ and $\lambda$. Then, using that $v\in L^{2^*}(0,\infty)$, \eqref{eq:g2g3} and Lemma~\ref{Talenti} , it is straightforward to check that the hypotheses of Gr{\"o}nwall's lemma (Lemma~\ref{Gronwall}) are satisfied, hence we obtain that
\begin{multline*}
\sqrt{\Psi(\tau)}\leq C\sqrt{\Psi_{|f|^2}(\tau)}+C\tau^{\frac{1}{n}}\mathcal{M}_g(\tau)+ Cv(\tau)\sqrt{\Psi_{|b|^2}(\tau)}\\+C\tau^{\frac{1}{n}-1}\int_0^{\tau}\sqrt{\Psi_{|f|^2}(\sigma)}\sqrt{\Psi_{|c|^2}(\sigma)}\exp\left(C\int_{\sigma}^{\tau}\rho^{\frac{1}{n}-1}\sqrt{\Psi_{|c|^2}(\rho)}\,d\rho\right)\,d\sigma\\
+C\tau^{\frac{1}{n}-1}\int_0^{\tau}\sigma^{\frac{1}{n}}\mathcal{M}_g(\sigma)\sqrt{\Psi_{|c|^2}(\sigma)}\exp\left(C\int_{\sigma}^{\tau}\rho^{\frac{1}{n}-1}\sqrt{\Psi_{|c|^2}(\rho)}\,d\rho\right)\,d\sigma\\
+C\tau^{\frac{1}{n}-1}\int_0^{\tau}v(\sigma)\sqrt{\Psi_{|b|^2}(\sigma)}\sqrt{\Psi_{|c|^2}(\sigma)}\exp\left(C\int_{\sigma}^{\tau}\rho^{\frac{1}{n}-1}\sqrt{\Psi_{|c|^2}(\rho)}\,d\rho\right)\,d\sigma,
\end{multline*}
where $C=C_{n,\lambda}$. Finally, using \eqref{eq:Talenti} to bound $\sqrt{\Psi}$ from below, and \eqref{eq:g2g3}, the proof is complete.
\end{proof}

\subsection{The maximum principle}

Using Lemma~\ref{MainEstimate}, we now show global boundedness of subsolutions.

\begin{prop}\label{GlobalUpperBound}
Let $\Omega\subseteq\bR^n$ be a domain. Let $A$ be uniformly elliptic and bounded in $\Omega$, with ellipticity $\lambda$. Let also $b,f\in L^{n,1}(\Omega)$, $d,g\in L^{\frac{n}{2},1}(\Omega)$, and suppose that $c=c_1+c_2\in L^{n,\infty}(\Omega)$ with $c_1\in L^{n,q}(\Omega)$ for some $q<\infty$ and $\|c_2\|_{n,\infty}<\nu$, where $\nu=\nu_{n,\lambda}$ appears in Lemma~\ref{MainEstimate}.

There exists $\tau_0\in(0,\infty)$, depending on $b,c_1,c_2$ and $d$ such that, for any subsolution $u\in Y^{1,2}(\Omega)$ to
\[
-\dive(A\nabla u+bu)+c\nabla u\leq -\dive f+g
\]
we have that
\begin{equation}\label{eq:forTau0}
\sup_{\Omega}u^+\leq C\sup_{\partial\Omega}u^++Cv(\tau_0)+C\|f\|_{n,1}+C\|g\|_{\frac{n}{2},1},
\end{equation}
where $C$ depends on $n,q,\lambda$, $\|b\|_{n,1}, \|c_1\|_{n,q}$ and $\|d\|_{\frac{n}{2},1}$. In particular, for any $p>0$,
\begin{equation}\label{eq:withTau}
\sup_{\Omega}u^+\leq C\sup_{\partial\Omega}u^++C\tau_0^{-1/p}\left(\int_{\Omega}|u^+|^p\right)^{\frac{1}{p}}+C\|f\|_{n,1}+C\|g\|_{\frac{n}{2},1}.
\end{equation}
\end{prop}
\begin{proof}
If $s=\sup_{\partial\Omega}u^+\in(0,\infty)$, then for every $s'>s$,
\[
-\dive(A\nabla(u-s')+b(u-s'))+c\nabla(u-s')+d(u-s')\leq -\dive(f-s'b)+g-s'd,
\]
and $(u-s')^+\in Y_0^{1,2}(\Omega)$; hence, we can assume that $s=0$, so $u^+\in Y_0^{1,2}(\Omega)$.

Consider the function $e$ from Lemma~\ref{Reduction} that solves the equation $\dive e=d$ in $\bR^n$. Then, if we define $b'=b-e$ and $c'=c-e$, $u$ is a subsolution to
\[
-\dive(A\nabla u+b'u)+c'\nabla u\leq -\dive f+g,
\]
Set $c_1'=c_1-e$, then $c'=c_1'+c_2$. Let $C_1,C_2$ be the constants in Lemma~\ref{MainEstimate}, and denote $C_1e^{C_2\|c_1'\|_{n,q}^q}$ by $C_0$. Moreover, set
\begin{equation}\label{eq:H}
\begin{split}
H(\tau)=C_1\tau^{\frac{1}{n}-1}\sqrt{\Psi_{|f|^2}(\tau)}+C_1\tau^{\frac{2}{n}-1}\mathcal{M}_g(\tau)+C_0\tau^{\frac{1}{n}-\frac{3}{2}}\int_0^{\tau}\sigma^{\frac{1}{n}-\frac{1}{2}}\sqrt{\Psi_{|f|^2}(\sigma)\Psi_{|c'|^2}(\sigma)}\,d\sigma\\
+C_0\tau^{\frac{1}{n}-\frac{3}{2}}\int_0^{\tau}\sigma^{\frac{2}{n}-\frac{1}{2}}\mathcal{M}_g(\sigma)\sqrt{\Psi_{|c'|^2}(\sigma)}\,d\sigma.
\end{split}
\end{equation}
From Lemma~\ref{Talenti}, we have that
\begin{equation}\label{eq:norms}
\left\|\sqrt{\Psi_{|f|^2}}\right\|_{n,1}\leq C_n\|f\|_{n,1},\qquad \left\|\sqrt{\Psi_{|c'|^2}}\right\|_{n,\infty}\leq C_n\|c'\|_{n,\infty}.
\end{equation}
Then, since $\frac{1}{n}-\frac{3}{2}<-1$, changing the order of integration and using \eqref{eq:Holder} and \eqref{eq:norms}, we have
\begin{align*}
\int_0^{\infty}\tau^{\frac{1}{n}-\frac{3}{2}}\int_0^{\tau}\sigma^{\frac{1}{n}-\frac{1}{2}}\sqrt{\Psi_{|f|^2}(\sigma)\Psi_{|c'|^2}(\sigma)}\,d\sigma\,d\tau&\leq C_n\int_0^{\infty}\tau^{\frac{2}{n}-1}\sqrt{\Psi_{|f|^2}(\sigma)\Psi_{|c'|^2}(\sigma)}\,d\sigma\\
&\leq C_n\|c'\|_{n,\infty}\|f\|_{n,1},
\end{align*}
and also, using \eqref{eq:maximalInequality}, we obtain that
\begin{align*}
\int_0^{\infty}\tau^{\frac{1}{n}-\frac{3}{2}}\int_0^{\tau}\sigma^{\frac{2}{n}-\frac{1}{2}}\mathcal{M}_{g}(\sigma)\sqrt{\Psi_{|c'|^2}(\sigma)}\,d\sigma&\leq C_n\int_0^{\infty}\sigma^{\frac{3}{n}-1}\mathcal{M}_{g}(\sigma)\sqrt{\Psi_{|c'|^2}(\sigma)}\,d\sigma\,d\tau\\
&\leq C_n\|c'\|_{n,\infty}\|g\|_{\frac{n}{2},1}.
\end{align*}
The last two estimates and the definition of $H$ in \eqref{eq:H} imply that
\begin{equation}\label{eq:HNorm}
\int_0^{\infty}H\leq C(\|c'\|_{n,\infty}+1)\left(\|f\|_{n,1}+\|g\|_{\frac{n}{2},1}\right),
\end{equation}
where $C$ depends on $n,q,\lambda$ and $\|c_1'\|_{n,q}$.

Set now
\[
R(\tau)=C_1\tau^{\frac{1}{n}-1}\sqrt{\Psi_{|b'|^2}(\tau)},\qquad G(\sigma)=C_1\sigma^{\frac{2}{n}-1}\sqrt{\Psi_{|b'|^2}(\sigma)\Psi_{|c'|^2}(\sigma)},
\]
Then, if $\|c_1'\|=\|c_1'\|_{n,q}$, \eqref{eq:Main} shows that
\[
-v'(\tau)\leq H(\tau)+R(\tau)v(\tau)+e^{C_2\|c'_1\|^q}\tau^{\frac{1}{n}-\frac{3}{2}}\int_0^{\tau}v(\sigma)\sigma^{\frac{1}{2}-\frac{1}{n}}G(\sigma)\,d\sigma,
\]
as long as $\|c_2\|_{n,\infty}<\nu_{n,\lambda}$. Since also $\int_0^{\infty}R\leq C_{n,\lambda}\|b'\|_{n,1}$ from Lemma~\ref{Talenti}, we obtain that
\begin{align}\label{eq:toIntegrate}
\begin{split}
-\left(e^{\int_0^{\tau}R}v\right)'&=e^{\int_0^{\tau}R}\left(-v'-Rv\right)\leq e^{\int_0^{\tau}R}\left(H(\tau)+e^{C_2\|c_1'\|^q}\tau^{\frac{1}{n}-\frac{3}{2}}\int_0^{\tau}v(\sigma)\sigma^{\frac{1}{2}-\frac{1}{n}}G(\sigma)\,d\sigma\right)\\
&\leq e^{C_{n,\lambda}\|b'\|_{n,1}}H(\tau)+ e^{C_{n,\lambda}\|b'\|_{n,1}+C_2\|c_1'\|^q}\tau^{\frac{1}{n}-\frac{3}{2}}\int_0^{\tau}v(\sigma)\sigma^{\frac{1}{2}-\frac{1}{n}}G(\sigma)\,d\sigma.
\end{split}
\end{align}
Set $B=\exp\left(C_{n,\lambda}\|b'\|_{n,1}\right)$ and $C'=\exp\left(C_{n,\lambda}\|b'\|_{n,1}+C_2\|c_1'\|^q\right)$, and let $\tau_2>\tau_1>0$. Then $v$ is absolutely continuous in $(\tau_1,\tau_2)$, from Lemma~\ref{uAbsCts}; hence, integrating \eqref{eq:toIntegrate} in $(\tau_1,\tau_2)$, we obtain that
\begin{align}\label{eq:exp}
\begin{split}
e^{\int_0^{\tau_1}R}v(\tau_1)&\leq e^{\int_0^{\tau_2}R}v(\tau_2)+B\int_{\tau_1}^{\tau_2}H+C'\int_{\tau_1}^{\tau_2}\int_0^{\tau}\tau^{\frac{1}{n}-\frac{3}{2}}v(\sigma)\sigma^{\frac{1}{2}-\frac{1}{n}}G(\sigma)\,d\sigma d\tau\\
&\leq B\left(v(\tau_2)+\|H\|_1\right)+C'\int_{\tau_1}^{\tau_2}\int_0^{\tau}\tau^{\frac{1}{n}-\frac{3}{2}}v(\sigma)\sigma^{\frac{1}{2}-\frac{1}{n}}G(\sigma)\,d\sigma d\tau.
\end{split}
\end{align}
Using Fubini's theorem, the last integral is equal to
\begin{multline*}
\int_0^{\tau_1}\int_{\tau_1}^{\tau_2}\tau^{\frac{1}{n}-\frac{3}{2}}\sigma^{\frac{1}{2}-\frac{1}{n}}G(\sigma)\,d\tau d\sigma+\int_{\tau_1}^{\tau_2}\int_{\sigma}^{\tau_2}\tau^{\frac{1}{n}-\frac{3}{2}}v(\sigma)\sigma^{\frac{1}{2}-\frac{1}{n}}G(\sigma)\,d\tau d\sigma\\
\leq C_n\tau_1^{\frac{1}{n}-\frac{1}{2}}\int_0^{\tau_1}v(\sigma)\sigma^{\frac{1}{2}-\frac{1}{n}}G(\sigma)\,d\sigma+C_n\int_{\tau_1}^{\tau_2}v(\sigma)G(\sigma)\,d\sigma,
\end{multline*}
therefore, plugging the last estimate in \eqref{eq:exp}, and using that $v$ is decreasing, we obtain 
\begin{equation}\label{eq:plug}
v(\tau_1)\leq B\left(v(\tau_2)+\|H\|_1\right)+C'C_n\tau_1^{\frac{1}{n}-\frac{1}{2}}\int_0^{\tau_1}v(\sigma)\sigma^{\frac{1}{2}-\frac{1}{n}}G(\sigma)\,d\sigma+C'C_nv(\tau_1)\int_{\tau_1}^{\tau_2}G(\sigma)\,d\sigma.
\end{equation}
Consider now $\tau_0>0$ such that
\begin{equation}\label{eq:Condition}
\int_0^{\tau_0}G\leq\frac{1}{2C'C_n};
\end{equation}
note that such $\tau_0$ always exists, since $G\in L^1(0,\infty)$ from Lemma~\ref{Talenti}. Then, if $0<\tau_1\leq\tau_0$, setting $\tau_2=\tau_0$ and plugging \eqref{eq:Condition} in \eqref{eq:plug} we obtain that
\[
v(\tau_1)\leq 2B\left(v(\tau_0)+\|H\|_1\right)+2C'C_n\tau_1^{\frac{1}{n}-\frac{1}{2}}\int_0^{\tau_1}v(\sigma)\sigma^{\frac{1}{2}-\frac{1}{n}}G(\sigma)\,d\sigma.
\]
Then, for $\tau_1\in(0,\tau_0)$, the hypotheses of Lemma~\ref{Gronwall} are satisfied, and we obtain that
\begin{align*}
v(\tau_1)&\leq 2B\left(v(\tau_0)+\|H\|_1\right)+2C'C_n\tau_1^{\frac{1}{n}-\frac{1}{2}}\int_0^{\tau_1}2B\left(v(\tau_0)+\|H\|_1\right)\sigma^{\frac{1}{2}-\frac{1}{n}}G(\sigma)e^{2C'C_n\int_{\sigma}^{\tau_1}G}\,d\sigma\\
&\leq 2B\left(v(\tau_0)+\|H\|_1\right)+4C'C_nB\left(v(\tau_0)+\|H\|_1\right)e^{2C'C_n\|G\|_1}\int_0^{\tau_1}G(\sigma)\,d\sigma.
\end{align*}
This estimate holds for every $0<\tau_1\leq\tau_0$, as long as \eqref{eq:Condition} holds. Then, letting $\tau_1\to 0^+$, and using the definition of $B$ and Lemma~\ref{Reduction}, we obtain that
\[
\lim_{\tau_1\to 0^+}v(\tau_1)\leq 2B\left(v(\tau_0)+\|H\|_1\right)\leq \exp\left(C_{n,\lambda}\left(\|b\|_{n,1}+\|d\|_{\frac{n}{2},1}\right)\right)\left(v(\tau_0)+\|H\|_1\right),
\]
as long as $\|c_2\|_{n,\infty}<\nu_{n,\lambda}$ and \eqref{eq:Condition} hold. Combining with \eqref{eq:HNorm} then shows \eqref{eq:forTau0}, and \eqref{eq:withTau} follows from the fact that $v$ is decreasing and \eqref{eq:equimeasurable}.
\end{proof}

As a corollary, we obtain the following maximum principle, which generalizes \cite[Proposition 7.5]{SakAPDE}. From Remark~\ref{bcShouldBeSmall}, to have such an estimate with constants depending only on the norms of the coefficients for arbitrary $b\in L^{n,1}$ requires that $c$ should have small norm; hence, we will assume that $c$ belongs to $L^{n,\infty}$ and has small norm.

\begin{prop}\label{MaxPrincipleC}
Let $\Omega\subseteq\bR^n$ be a domain. Let $A$ be uniformly elliptic and bounded in $\Omega$, with ellipticity $\lambda$, and let $b,f\in L^{n,1}(\Omega)$, $g\in L^{\frac{n}{2},1}(\Omega)$, with $\|b\|_{n,1}\leq M$.

There exists $\beta=\beta_{n,\lambda,M}>0$ such that, if $c\in L^{n,\infty}(\Omega)$ and $d\in L^{\frac{n}{2},1}(\Omega)$ with $\|c\|_{n,\infty}<\beta$ and $\|d\|_{\frac{n}{2},1}<\beta$, then for every subsolution $u\in Y^{1,2}(\Omega)$ to
\[
-\dive(A\nabla u+bu)+c\nabla u+du\leq -\dive f+g
\]	
in $\Omega$, we have that
\[
\sup_{\Omega}u\leq C\sup_{\partial\Omega}u^++C\|f\|_{n,1}+C\|g\|_{\frac{n}{2},1},
\]
where $C$ depends on $n,\lambda$ and $M$.
\end{prop}
\begin{proof}
Assume that $\|c\|_{n,\infty}<\beta$ and $\|d\|_{\frac{n}{2},1}<\beta$, for $\beta$ to be chosen later. Consider the $\nu_{n,\lambda}$ from Lemma~\ref{MainEstimate}, and take $c_1\equiv 0$ and $q=1$ in Proposition~\ref{GlobalUpperBound}. We will take $\beta\leq\nu_{n,\lambda}$, so it is enough to show that we can take $\tau_0=\infty$ in \eqref{eq:forTau0}, since $\lim_{\tau\to\infty}v(\tau)=0$. Hence, from \eqref{eq:Condition}, and the definitions of $C'$ and $e$ from the proof of Proposition~\ref{GlobalUpperBound}, it will be enough to have that
\begin{equation}\label{eq:ToSat}
\int_0^{\infty}\sigma^{\frac{2}{n}-1}\sqrt{\Psi_{|b-e|^2}(\sigma)\Psi_{|c-e|^2}(\sigma)}\,d\sigma\leq C\exp\left(-C\|b-e\|_{n,1}-C\|e\|_{n,1}\right),
\end{equation}
where $C$ depends on $n$ and $\lambda$ only.

We first bound the left hand side from above using Lemmas~\ref{Talenti} and \ref{Reduction}, to obtain
\begin{align*}
\int_0^{\infty}\sigma^{\frac{2}{n}-1}\sqrt{\Psi_{|b-e|^2}(\sigma)\Psi_{|c-e|^2}(\sigma)}\,d\sigma&=\left\|\sqrt{\Psi_{|b-e|^2}\Psi_{|c-e|^2}}\right\|_{\frac{n}{2},1}\leq C\left\|\sqrt{\Psi_{|b-e|^2}}\right\|_{n,1}\left\|\sqrt{\Psi_{|c-e|^2}}\right\|_{n,\infty}\\
&\leq C\|b-e\|_{n,1}\|c-e\|_{n,\infty}\leq C(M+\beta)\beta,
\end{align*}
while
\[
-C\|b-e\|_{n,1}-C\|e\|_{n,1}\geq-C\|b\|_{n,1}-C\|e\|_{n,1}\geq -CM-C\beta.
\]
From the last two estimates, \eqref{eq:ToSat} will be satisfied as long as
\[
C(M+\beta)\beta e^{CM+C\beta}\leq 1.
\]
So, choosing $\beta>0$ depending on $n,\lambda$ and $M$, such that the last estimate is satisfied and also $\beta\leq\nu_{n,\lambda}$ completes the proof.
\end{proof}

In addition, we also obtain the following maximum principle, which concerns perturbations of the operator $-\dive(A\nabla u)+c\nabla u$.

\begin{prop}\label{MaxPrincipleB}
Let $\Omega\subseteq\bR^n$ be a domain. Let $A$ be uniformly elliptic and bounded in $\Omega$, with ellipticity $\lambda$, and let $q<\infty$, $c=c_1+c_2\in L^{n,\infty}(\Omega)$, with $\|c_2\|_{n,\infty}<\nu$ and $\|c_1\|_{n,q}\leq M$, where $\nu=\nu_{n,\lambda}$ appears in Lemma~\ref{MainEstimate}. Assume also that $f\in L^{n,1}(\Omega)$, $g\in L^{\frac{n}{2},1}(\Omega)$. 

There exists $\gamma=\gamma_{n,q,\lambda,M}>0$ such that, if $b\in L^{n,1}(\Omega)$ and $d\in L^{\frac{n}{2},1}(\Omega)$ with $\|b\|_{n,1}<\gamma$ and $\|d\|_{\frac{n}{2},1}<\gamma$, then for any subsolution $u\in Y^{1,2}(\Omega)$ to
\[
-\dive(A\nabla u+bu)+c\nabla u+du\leq -\dive f+g
\]	
in $\Omega$, we have
\[
\sup_{\Omega}u\leq C\sup_{\partial\Omega}u^++C\|f\|_{n,1}+C\|g\|_{\frac{n}{2},1},
\]
where $C$ depends on $n,q,\lambda$ and $M$.
\end{prop}
\begin{proof}
As in the proof of Corollary~\ref{MaxPrincipleC}, we will take $\gamma\leq\nu_{n,\lambda}$, and it will be enough to have that
\[
\int_0^{\infty}\sigma^{\frac{2}{n}-1}\sqrt{\Psi_{|b-e|^2}(\sigma)\Psi_{|c-e|^2}(\sigma)}\,d\sigma\leq C\exp\left(-C\|b-e\|_{n,1}-C\|c_1-e\|_{n,q}^q\right),
\]
whenever $\|b\|_{n,1}<\gamma$ and $\|d\|_{\frac{n}{2},1}<\gamma$, and where $C=C_{n,q,\lambda}$. Then, a similar argument as in the proof of Proposition~\ref{MaxPrincipleC} completes the proof.
\end{proof}

\section{Local boundedness}\label{secLocal}

\subsection{The first step: all coefficients are small}

The first step to obtain the Moser estimate is via a coercivity assumption, which we now turn to. The following lemma is standard, and we only give a sketch of its proof. We will set $2_*=\frac{2n}{n+2}$.

\begin{lemma}\label{Coercivity}
Let $\Omega\subseteq\bR^n$ be a domain, and $A$ be uniformly elliptic and bounded in $\Omega$, with ellipticity $\lambda$. There exists $\theta=\theta_{n,\lambda}>0$ such that, if $b\in L^{n,1}(\Omega)$, $c\in L^{n,\infty}(\Omega)$ and $d\in L^{\frac{n}{2},1}(\Omega)$ with $\|b\|_{n,1}\leq\theta$, $\|c\|_{n,\infty}\leq\theta$ and $\|d\|_{\frac{n}{2},1}\leq\theta$, then the operator
\[
\mathcal{L}u=-\dive(A\nabla u+bu)+c\nabla u+du
\]
is coercive, and every solution $v\in W_0^{1,2}(\Omega)$ to the equation $\mathcal{L}u=-\dive F+G$ for $F\in L^2(\Omega)$ and $G\in L^{2_*}(\Omega)$ satisfies the estimate
\begin{equation}\label{eq:SecondInLemma}
\|\nabla v\|_{L^2(B_{2r})}\leq C_{n,\lambda}\|F\|_{L^2(\Omega)}+C_{n,\lambda}\|G\|_{L^{2_*}(\Omega)}.
\end{equation}
Also, if $\Omega=B_{2r}$ and $w\in W^{1,2}(B_{2r})$ is a subsolution to $-\dive(A\nabla w+bw)+c\nabla w+dw\leq 0$, then
\begin{equation}\label{eq:Cacciopoli}
\int_{B_r}|\nabla w|^2\leq\frac{C}{r^2}\int_{B_{2r}}|w^+|^2,
\end{equation}
where $C$ depends on $n,\lambda$ and $\|A\|_{\infty}$.

Moreover, for any subsolution $u\in W^{1,2}(B_{2r})$ to $-\dive(A\nabla u)+c\nabla u\leq 0$ in $B_{2r}$ and $\alpha\in(1,2)$, we have that
\begin{equation}\label{eq:FirstInLemma}
\sup_{B_r}u\leq\frac{C}{(\alpha-1)^{n/2}}\left(\fint_{B_{\alpha r}}|u^+|^2\right)^{\frac{1}{2}},
\end{equation}
where $C$ depends on $n,\lambda$ and $\|A\|_{\infty}$.
\end{lemma}
\begin{proof}
We first show \eqref{eq:FirstInLemma}, following the lines of the proof of \cite[Theorem 8.17]{Gilbarg}: if $\phi$ is a smooth cutoff function, then using $u^+\phi^2$ as a test function, we obtain 
\begin{align}\label{eq:A}
\begin{split}
\int_{B_{2r}}A\nabla u^+\nabla u^+\cdot \phi^2&\leq -2\int_{B_{2r}}A\nabla u^+\nabla\phi\cdot u^+\phi-\int_{B_{2r}}c\nabla u^+\cdot u^+\phi^2\\
&\leq C\|\phi\nabla u^+\|_{L^2(B_{2r})}\|u^+\nabla\phi\|_{L^2(B_{2r})}+\left\|cu^+\phi\right\|_{L^2(B_{2r})}\|\phi\nabla u^+\|_{L^2(B_{2r})}.
\end{split}
\end{align}
Then, assuming that $\|c\|_{n,\infty}\leq\theta$, for $\theta$ to be chosen later, using H{\"o}lder's estimate \eqref{eq:Holder} we have
\begin{align*}
\left\|cu^+\phi\right\|_{L^2(B_{2r})}\leq C_n\|c\|_{n,\infty}\|u^+\phi\|_{L^{2^*,2}(B_{2r})}\leq C_n\theta\|u^+\phi\|_{L^{2^*,2}(B_{2r})},
\end{align*}
and combining with \cite[Lemma 2.2]{SakAPDE}, we have
\begin{equation}\label{eq:cBound}
\left\|cu^+\phi\right\|_{L^2(B_{2r})}\leq C_n\theta\|\nabla(u^+\phi)\|_{L^2(B_{2r})}\leq C_n\theta\|\phi\nabla u^+\|_{L^2(B_{2r})}+C_n\theta\|u^+\nabla \phi\|_{L^2(B_{2r})}.
\end{equation}
So, choosing $\theta$ such that $C_n\theta<\frac{\lambda}{4}$, and plugging in \eqref{eq:A}, we obtain that
\[
\int_{B_{2r}}|\phi\nabla u^+|^2\leq C\int_{B_{2r}}|u^+\nabla\phi|^2,
\]
where $C$ depends on $n,\lambda$ and $\|A\|_{\infty}$. This estimate corresponds to \cite[(8.53), page 196]{Gilbarg}, and following the lines of the argument on \cite[pages 196 and 197]{Gilbarg} we obtain that
\[
\sup_{B_r}u\leq C\left(\fint_{B_{2r}}|u^+|^2\right)^{\frac{1}{2}},
\]
where $C$ depends on $n,\lambda$ and $\|A\|_{\infty}$. To complete the proof of \eqref{eq:FirstInLemma} note that, for all $x\in B_r$, the last estimate shows that
\[
\sup_{B_{\frac{\alpha-1}{2}r}(x)}u\leq C\left(\fint_{B_{(\alpha-1)r}(x)}|u^+|^2\right)^{\frac{1}{2}}\leq \frac{C}{(\alpha-1)^{n/2}r^{n/2}}\left(\int_{B_{\alpha r}}|u^+|^2\right)^{\frac{1}{2}},
\]
since $B_{(\alpha-1)r}(x)\subseteq B_{\alpha r}$, and considering the supremum for $x\in B_r$ shows \eqref{eq:FirstInLemma}.
	
Finally, coercivity of $\mathcal{L}$, \eqref{eq:SecondInLemma} and \eqref{eq:Cacciopoli} follow via a combination of the procedure as in \eqref{eq:A} and \eqref{eq:cBound}, where for \eqref{eq:SecondInLemma} we use $v$ as a test function, and for \eqref{eq:Cacciopoli} we use $w^+\phi^2$ as a test function.
\end{proof}

We now turn to local boundedness when all the lower order coefficients have small norms.

\begin{lemma}\label{localAllSmall}
Let $A$ be uniformly elliptic and bounded in $B_{2r}$, with ellipticity $\lambda$. There exists $\theta'=\theta'_{n,\lambda}>0$ such that, if $b\in L^{n,1}(B_{2r})$, $c\in L^{n,\infty}(B_{2r})$ and $d\in L^{\frac{n}{2},1}(B_{2r})$ with $\|b\|_{n,1}\leq\theta'$, $\|c\|_{n,\infty}\leq\theta'$ and $\|d\|_{\frac{n}{2},1}\leq\theta'$, then for any subsolution $u\in W^{1,2}(B_{2r})$ to $-\dive(A\nabla u+bu)+c\nabla u+du\leq 0$,
\[
\sup_{B_r}u^+\leq C\left(\fint_{B_{2r}}|u^+|^2\right)^{\frac{1}{2}},
\]
where $C$ depends on $n,\lambda$ and $\|A\|_{\infty}$.
\end{lemma}
\begin{proof}
Consider the $\theta_{n,\lambda}$ that appears in Lemma~\ref{Coercivity}. We will take $\theta'\leq\theta_{n,\lambda}$, so that the operator is coercive. Then, if $u$ is a subsolution to $\mathcal{L}u\leq 0$, the proof of \cite[Theorem 3.5]{StampacchiaDirichlet} implies that $u^+$ is a subsolution to $\mathcal{L}u^+\leq 0$; therefore, we can assume that $u\geq 0$.

Assume first that $b,c,d$ are bounded in $B_{2r}$, then \cite[Theorem 8.17]{Gilbarg} shows that
\begin{equation}\label{eq:uBounded}
\sup_{B_r}u<\infty.
\end{equation}
Let $\frac{1}{4}\leq\eta<\eta'\leq\frac{1}{2}$. From coercivity of the operator $\mathcal{L}_0u=-\dive(A\nabla u)+c\nabla u$, and since $\dive(bu)-du\in W^{-1,2}(B_{\eta'r})=\left(W_0^{1,2}(B_{\eta'r})\right)^*$, the Lax-Milgram theorem shows that there exists $v\in W_0^{1,2}(B_{\eta' r})$ such that
\[
-\dive(A\nabla v)+c\nabla v=\dive(bu)-du.
\]
If $\beta$ is as in Proposition~\ref{MaxPrincipleC}, taking $\theta'\leq\beta_{n,\lambda,\theta_{n,\lambda}}$ the same proposition shows that
\begin{equation}\label{eq:supBound}
\sup_{B_{\eta' r}}v\leq C_{n,\lambda}\|bu\|_{L^{n,1}(B_{\eta' r})}+C_{n,\lambda}\|du\|_{L^{{\frac{n}{2},1}}(B_{\eta' r})}\leq C_{n,\lambda}\theta'\sup_{B_{\eta' r}}u,
\end{equation}
since $u\geq 0$. In addition, from the Sobolev inequality, estimate \eqref{eq:SecondInLemma} and the H{\"o}lder inequality,
\begin{equation}\label{eq:vAbove}
\|v\|_{L^{2^*}(B_{\eta' r})}\leq C_n\|\nabla v\|_{L^2(B_{\eta' r})}\leq C_{n,\lambda}\|bu\|_{L^2(B_{\eta' r})}+C_{n,\lambda}\|du\|_{L^{2_*}(B_{\eta' r})}\leq C_{n,\lambda}\|u\|_{L^{2^*}(B_{\eta' r})}.
\end{equation}
Moreover, the function $w=u-v$ is a subsolution to $-\dive(A\nabla w)+c\nabla w\leq 0$, so \eqref{eq:FirstInLemma} implies that
\begin{align*}
\sup_{B_{\eta r}}w&\leq\frac{C}{(\frac{\eta'}{\eta}-1)^{n/2}}\left(\fint_{B_{\eta'r}}|w^+|^2\right)^{\frac{1}{2}}\leq\frac{C}{(\eta'-\eta)^{n/2}}\left(\fint_{B_{\eta' r}}|u|^2\right)^{\frac{1}{2}}+\frac{C}{(\eta'-\eta)^{n/2}}\left(\fint_{B_{\eta' r}}|v|^2\right)^{\frac{1}{2}}\\
&\leq\frac{C}{(\eta'-\eta)^{n/2}}\left(\fint_{B_{\eta' r}}|u|^{2^*}\right)^{\frac{1}{2^*}}\leq\frac{C}{(\eta'-\eta)^{n/2}}\left(\fint_{B_{r/2}}|u|^{2^*}\right)^{\frac{1}{2^*}},
\end{align*}
where we also used \eqref{eq:vAbove} for the penultimate estimate, and $C$ depends on $n,\lambda$ and $\|A\|_{\infty}$. Hence, the definition of $w$, the last estimate and \eqref{eq:supBound} show that
\[
\sup_{B_{\eta r}}u\leq\sup_{B_{\eta r}}v+\sup_{B_{\eta r}}w\leq C_{n,\lambda}\theta'\sup_{B_{\eta'r}}u+\frac{C}{(\eta'-\eta)^{n/2}}\left(\fint_{B_{r/2}}|u|^{2^*}\right)^{\frac{1}{2^*}},
\]
where $C$ depends on $n,\lambda$ and $\|A\|_{\infty}$.

We now set $\eta_N=\frac{1}{2}-4^{-N}$ and apply the previous estimate for $\eta=\eta_N$ and $\eta'=\eta_{N+1}$. Then,
\[
\sup_{B_{\eta_Nr}}u\leq C_{n,\lambda}\theta'\sup_{B_{\eta_{N+1} r}}u+2^{nN}C\left(\fint_{B_{r/2}}|u|^{2^*}\right)^{\frac{1}{2^*}}.
\]
Inductively, this shows that, for any $N\in\mathbb N$,
\[
\sup_{B_{\eta_1r}}u\leq (C_{n,\lambda}\theta')^N\sup_{B_{\eta_{N+1}r}}u+C\sum_{i=1}^N(C_{n,\lambda}\theta')^{i-1}2^{ni}\cdot \left(\fint_{B_r}|u|^{2^*}\right)^{\frac{1}{2^*}}.
\]
We will consider $\theta'$ such that $C_{n,\lambda}\theta'\leq\frac{1}{2}$. Then, letting $N\to\infty$ and using \eqref{eq:uBounded}, we obtain that
\[
\sup_{B_{r/4}}u\leq C\sum_{i=1}^{\infty}\left(2^nC_{n,\lambda}\theta'\right)^{i-1}\cdot \left(\fint_{B_{r/2}}|u|^{2^*}\right)^{\frac{1}{2^*}},
\]
and choosing $\theta'$ that also satisfies $2^nC_{n,\lambda}\theta'\leq\frac{1}{2}$ shows that
\begin{equation}\label{eq:forBounded}
\sup_{B_{r/4}}u\leq C \left(\fint_{B_{r/2}}|u|^{2^*}\right)^{\frac{1}{2^*}}\leq C\left(\fint_{B_{2r}}|u|^2\right)^{\frac{1}{2}},
\end{equation}
where we used \eqref{eq:Cacciopoli} and the Sobolev inequality for the last estimate, and where $C$ depends on $n,\lambda$ and $\|A\|_{\infty}$.

In the case that $b,c,d$ are not necessarily bounded, let $b^j$ be the coordinate functions of $b$, and define $b_N$ having coordinate functions $b_N^j=b^j\chi_{[|b^j|\leq N]}$ for $N\in\mathbb N$; define also similar approximations $c_N$ and $d_N$ for $c,d$ respectively. We then have that $\|b_N\|_{n,1}\leq \|b\|_{n,1}$, and similarly for $c_N$ and $d_N$. Since $\theta'\leq\theta_{n,\lambda}$, from coercivity in Lemma~\ref{Coercivity} and the Lax-Milgram theorem there exists $v_N\in W_0^{1,2}(B_{r/2})$ that solves the equation
\[
-\dive(A\nabla v_N+b_Nv_N)+c_N\nabla v_N+d_Nv_N=-\dive(A\nabla u+b_Nu)+c_N\nabla u+d_Nu
\]
in $B_{r/2}$. Then, from \eqref{eq:SecondInLemma},
\begin{equation}\label{eq:nablaVn}
\|\nabla v_n\|_{L^2(B_{r/2})}\leq C\|A\nabla u+b_Nu\|_{L^2(B_{r/2})}+C\|c_N\nabla u+d_Nu\|_{L^{2_*}(B_{r/2})}\leq\frac{C}{r}\|u\|_{L^2(B_{2r})},
\end{equation}
where we also used \eqref{eq:Cacciopoli} and H{\"o}lder's inequality for the last estimate. So, $(v_N)$ is bounded in $W_0^{1,2}(B_{r/2})$, hence from Rellich's theorem there exists a subsequence $(v_{N'})$ such that
\begin{equation}\label{eq:Convergences}
v_{N'}\to v_0\,\,\,\text{weakly in}\,\,\,W_0^{1,2}(B_{r/2})\,\,\,\text{and strongly in}\,\,\,L^{\frac{n}{n-2}}(B_{r/2}),\quad v_{N'}(x)\to v_0(x)\,\,\,\forall\,x\in F,
\end{equation}
where $F\subseteq B_{r/2}$ is a set with full measure.

Note now that $w_N=u-v_N$ is a solution to $-\dive(A\nabla w_N+b_Nw_N)+c_N\nabla w_N+d_Nw_N=0$ in $B_{r/2}$, and $b_N,c_N$ and $d_N$ are bounded, so \eqref{eq:forBounded} (where $B_{2r}$ is replaced by $B_{r/2}$) is applicable to $w^+$; therefore, for $x\in F_N$, where $F_N\subseteq B_{r/16}$ has full measure,
\[
w_N^+(x)\leq\sup_{B_{r/16}}w_N^+\leq C\left(\fint_{B_{r/2}}|w_N^+|^2\right)^{\frac{1}{2}}\leq C\left(\fint_{B_{r/2}}u^2\right)^{\frac{1}{2}}+C\left(\fint_{B_{r/2}}v_N^2\right)^{\frac{1}{2}},
\]
where $C$ depends on $n,\lambda$ and $\|A\|_{\infty}$. Therefore, for all $x\in F_N$,
\[
u(x)=v_N(x)+w_N(x)\leq v_N(x)+C\left(\fint_{B_{r/2}}u^2\right)^{\frac{1}{2}}+C\left(\fint_{B_{r/2}}v_N^2\right)^{\frac{1}{2}}\leq v_N(x)+C\left(\fint_{B_{2r}}u^2\right)^{\frac{1}{2}},
\]
where we used the Sobolev inequality and \eqref{eq:nablaVn} for the last estimate.

Let now $F_0=F\cap\bigcap_{N=1}^{\infty}F_N$, then $F_0\subseteq B_{r/16}$ has full measure, and if $x\in F_0$, then letting $N'\to\infty$ in the previous estimate, \eqref{eq:Convergences} implies that
\begin{equation}\label{eq:v0}
u(x)\leq\limsup_{N'\to\infty}v_{N'}(x)+C\left(\fint_{B_{2r}}u^2\right)^{\frac{1}{2}}=v_0(x)+C\left(\fint_{B_{2r}}u^2\right)^{\frac{1}{2}},
\end{equation}
for all $x\in F_0$. Finally, note that $v_{N'}$ is a subsolution to
\[
-\dive(A\nabla v_N+b_Nv_N)+c_N\nabla v_N+d_Nv_N\leq -\dive((b_N-b)u)+(c_N-c)\nabla u+(d_N-d)u
\]
in $B_{r/2}$, and since $b_{N'}\to b_N$ and $c_{N'}\to c_N$ strongly in $L^2(B_{r/2})$, while $d_{N'}\to d_N$ strongly in $L^{\frac{n}{2}}(B_{r/2})$, using \eqref{eq:Convergences} and the variational formulation of subsolutions \eqref{eq:subsolDfn} we obtain that $v_0$ is a $W_0^{1,2}(B_{2r})$ subsolution to
\[
-\dive(A\nabla v_0+bv_0)+c\nabla v_0+dv_0\leq 0.
\]
Hence, since $\theta'\leq\beta_{n,\lambda,\theta_{n,\lambda}}$, Proposition~\ref{MaxPrincipleC} implies that $v_0\leq 0$ in $B_{r/2}$, and plugging in \eqref{eq:v0} and covering $B_r$ with balls of radius $r/16$ completes the proof.
\end{proof}

\subsection{The second step: \texorpdfstring{$b$}{b} or \texorpdfstring{$c$}{c} have large norms}

We now turn to scale invariant estimates with ``good" constants when $d$ is small, and either $b$ or $c$ are small as well. We first consider the case of small $c$ and assume that the right hand side is identically $0$, for simplicity; the terms on the right hand side will be added in Proposition~\ref{MoserB}.

\begin{lemma}\label{LocalBoundSmallc}
Let $A$ be uniformly elliptic and bounded in $B_{2r}$, with ellipticity $\lambda$, and $b\in L^{n,1}(B_{2r})$ with $\|b\|_{n,1}\leq M$.

There exists $\overline{\theta}=\overline{\theta}_{n,\lambda,M}>0$ such that, if $c\in L^{n,\infty}(B_{2r})$ and $d\in L^{\frac{n}{2},1}(B_{2r})$ with $\|c\|_{n,\infty}<\overline{\theta}$ and $\|d\|_{\frac{n}{2},1}<\overline{\theta}$, then for any subsolution $u\in W^{1,2}(B_{2r})$ to $-\dive(A\nabla u+bu)+c\nabla u+du\leq 0$, we have
\begin{equation}\label{eq:localLargeBSquare}
\sup_{B_r}u\leq C\left(\fint_{B_{2r}}|u^+|^2\right)^{\frac{1}{2}},
\end{equation}
where $C$ depends on $n,\lambda,\|A\|_{\infty}$ and $M$.
\end{lemma}
\begin{proof}
We will proceed by induction on $M$. Consider the $\theta_{n,\lambda}'$ and the constant $C_0=C_{n,\lambda,\|A\|_{\infty}}\geq 1$ that appear in Lemma~\ref{localAllSmall}. In addition, for any integer $N\geq 0$, set $C'_{n,\lambda,N}=C_{n,\lambda,2^{N/n}\theta'_{n,\lambda}}\geq 1$, where the last constant appears in Proposition~\ref{MaxPrincipleC}.

We claim that, if $\|b\|_{L^{n,1}(B_{2r})}\leq 2^{N/n}\theta_{n,\lambda}'$, then there exists $\overline{\theta}_{n,\lambda,N}>0$ such that, if we have that $\|c\|_{L^{n,\infty}(B_{2r})}<\overline{\theta}_{n,\lambda,N}$ and $\|d\|_{L^{\frac{n}{2},1}(B_{2r})}<\overline{\theta}_{n,\lambda,N}$, then
\begin{equation}\label{eq:forInduction}
\sup_{B_r}u\leq 8^{\frac{nN}{2}}C_0\prod_{i=0}^NC'_{n,\lambda,i}\left(\fint_{B_{2r}}|u^+|^2\right)^{\frac{1}{2}}.
\end{equation}
For $N=0$, letting $\overline{\theta}_{n,\lambda,0}=\theta'_{n,\lambda}$, the previous estimate holds from Lemma~\ref{localAllSmall}.

Assume now that this estimate holds for some integer $N\geq 0$, for some constant $\overline{\theta}_{n,\lambda,N}$. From Proposition~\ref{MaxPrincipleC} there exists $\beta'_{n,\lambda,N}=\beta_{n,\lambda,2^{N/n}\theta'_{n,\lambda}}>0$ such that, if $\Omega\subseteq\bR^n$ is a domain, $A'$ is elliptic in $\Omega$ with ellipticity $\lambda$, $\|b'\|_{L^{n,1}(\Omega)}\leq 2^{N/n}\theta'_{n,\lambda}$, $\|c'\|_{L^{n,\infty}(\Omega)}<\beta'_{n,\lambda,N}$ and $\|d'\|_{L^{\frac{n}{2},1}(\Omega)}<\beta'_{n,\lambda,N}$, then for any subsolution $v\in Y^{1,2}(\Omega)$ to $-\dive(A'\nabla v+b'v)+c'\nabla v+d'v\leq 0$ in $\Omega$, we have that
\[
\sup_{\Omega}v\leq C'_{n,\lambda,N}\sup_{\partial\Omega}v^+.
\]
We then set $\overline{\theta}_{n,\lambda,{N+1}}=\min\{\overline{\theta}_{n,\lambda,N},\beta'_{n,\lambda,N+1}\}$, and assume that
\begin{equation}\label{eq:inductionAssumptions}
\|b\|_{L^{n,1}(B_{2r})}\leq 2^{(N+1)/n}\theta'_{n,\lambda},\qquad \|c\|_{L^{n,\infty}(B_{2r})}<\overline{\theta}_{n,\lambda,{N+1}},\quad\text{and}\quad\|d\|_{L^{\frac{n}{2},1}(B_{2r})}<\overline{\theta}_{n,\lambda,{N+1}}.
\end{equation}
We will show that, in this case, \eqref{eq:forInduction} holds for $N+1$. To show this, we distinguish between two cases: $\|b\|_{L^{n,1}(B_{3r/2})}\leq 2^{N/n}\theta'_{n,\lambda}$, and $\|b\|_{L^{n,1}(B_{3r/2})}>2^{N/n}\theta'_{n,\lambda}$.

In the first case, let $x\in B_r$. Then, since $\overline{\theta}_{n,\lambda,{N+1}}\leq\overline{\theta}_{n,\lambda,N}$ and $B_{r/2}(x)\subseteq B_{3r/2}$, we have that
\[
\|b\|_{L^{n,1}(B_{r/2}(x))}\leq 2^{N/n}\theta'_{n,\lambda},\qquad \|c\|_{L^{n,\infty}(B_{r/2}(x))}<\overline{\theta}_{n,\lambda,N},\quad\text{and}\quad\|d\|_{L^{\frac{n}{2},1}(B_{r/2}(x))}<\overline{\theta}_{n,\lambda,N}.
\]
Therefore, from \eqref{eq:forInduction} for $N$ (in the ball $B_{r/2}(x)$ instead of $B_{2r}$), we have
\begin{align*}
\sup_{B_{r/4}(x)}u&\leq 8^{\frac{nN}{2}}C_0\prod_{i=0}^NC'_{n,\lambda,i}\left(\fint_{B_{r/2}(x)}|u^+|^2\right)^{\frac{1}{2}}\\
&\leq 8^{\frac{nN}{2}}C_0\prod_{i=0}^NC'_{n,\lambda,i}2^n\left(\fint_{B_{2r}}|u^+|^2\right)^{\frac{1}{2}}\leq 8^{\frac{n(N+1)}{2}}C_0\prod_{i=0}^{N+1}C'_{n,\lambda,i}\left(\fint_{B_{2r}}u^2\right)^{\frac{1}{2}},
\end{align*}
where we used that $C'_{n,\lambda,N+1}\geq 1$ for the last step. So, \eqref{eq:forInduction} holds for $N+1$ in this case.

In the second case, let $y\in\partial B_{7r/4}$. Then $B_{r/4}(y)\subseteq B_{2r}\setminus B_{3r/2}$, therefore, from Lemma~\ref{NormDisjoint},
\[
\|b\|_{L^{n,1}(B_{r/4}(y))}^n\leq \|b\|_{L^{n,1}(B_{2r})}^n-\|b\|_{L^{n,1}(B_{3r/2})}^n< 2^{N+1}(\theta'_{n,\lambda})^n-2^N(\theta'_{n,\lambda})^n=(2^{N/n}\theta'_{n,\lambda})^n.
\]
Moreover, from \eqref{eq:inductionAssumptions}, we have that $\|c\|_{L^{n,\infty}(B_{r/4}(y))}<\overline{\theta}_{n,\lambda,N}$ and $\|d\|_{L^{\frac{n}{2},1}(B_{r/4}(y))}<\overline{\theta}_{n,\lambda,N}$, hence \eqref{eq:forInduction} for $N$ (in the ball $B_{r/4}(y)$ instead of $B_{2r}$) implies that
\[
\sup_{B_{r/8}(y)}u\leq 8^{\frac{nN}{2}}C_0\prod_{i=0}^NC'_{n,\lambda,i}\left(\fint_{B_{r/4}(y)}|u^+|^2\right)^{\frac{1}{2}}\leq 8^{\frac{n(N+1)}{2}}C_0\prod_{i=0}^NC'_{n,\lambda,i}\left(\fint_{B_{2r}}u^2\right)^{\frac{1}{2}}.
\]
Then, the last estimate, \eqref{eq:inductionAssumptions} and Proposition~\ref{MaxPrincipleC} show that
\[
\sup_{B_r}u\leq C'_{n,\lambda,N+1}\sup_{\partial B_{7r/4}}u\leq C_{n,\lambda,N+1}'\cdot 8^{\frac{n(N+1)}{2}}C_0\prod_{i=1}^NC'_{n,\lambda,i}\left(\fint_{B_{2r}}|u^+|^2\right)^{\frac{1}{2}},
\]
which shows that \eqref{eq:forInduction} for $N+1$ in this case as well.

Therefore, \eqref{eq:forInduction} holds for any $N\in\mathbb N$, which completes the proof.
\end{proof}

Finally, we show Moser's estimate allowing right hand sides to the equation, and considering also different $L^p$ norms on the right hand side of the estimate.

\begin{prop}\label{MoserB}
Let $A$ be uniformly elliptic and bounded in $B_{2r}$, with ellipticity $\lambda$. Let also $b\in L^{n,1}(B_{2r})$ with $\|b\|_{n,1}\leq M$, and $p>0$, $f\in L^{n,1}(B_{2r})$, $g\in L^{\frac{n}{2},1}(B_{2r})$.

There exists $\e=\e_{n,\lambda,M}>0$ such that, if $c\in L^{n,\infty}(B_{2r})$ and $d\in L^{\frac{n}{2},1}(B_{2r})$ with $\|c\|_{n,\infty}<\e$ and $\|d\|_{\frac{n}{2},1}<\e$, then for any subsolution $u\in W^{1,2}(B_{2r})$ to $-\dive(A\nabla u+bu)+c\nabla u+du\leq -\dive f+g$, we have that
\begin{equation}\label{eq:MoserForB}
\sup_{B_r}u\leq C\left(\fint_{B_{2r}}|u^+|^p\right)^{\frac{1}{p}}+C\|f\|_{L^{n,1}(B_{2r})}+C\|g\|_{L^{\frac{n}{2},1}(B_{2r})},
\end{equation}
where $C$ depends on $n,p,\lambda,\|A\|_{\infty}$ and $M$.
\end{prop}
\begin{proof}
Consider the $\beta_{n,\lambda,M}$ from Proposition~\ref{MaxPrincipleC}. If $\|c\|_{n,\infty}<\beta_{n,\lambda,M}$ and $\|d\|_{\frac{n}{2},1}<\beta_{n,\lambda,M}$, any solution $u\in W_0^{1,2}(B_{2r})$ to the equation $-\dive(A\nabla u+bu)+c\nabla u+du=0$ in $B_{2r}$ should be identically $0$, from Proposition~\ref{MaxPrincipleC}. Hence, adding a term of the form $+Lu$ to the operator, for some large $L>0$ depending only on $n,\lambda,M$, the operator becomes coercive, and a combination of the Lax-Milgram theorem and the Fredholm alternative (as in \cite[Theorem 4, pages 303-305]{Evans}, for example) show that there exists a unique $v\in W_0^{1,2}(B_{2r})$ such that
\[
-\dive(A\nabla v+bv)+c\nabla v+dv=-\dive f+g,
\]
in $B_{2r}$. Then, Proposition~\ref{MaxPrincipleC} implies that
\begin{equation}\label{eq:sup1}
\sup_{B_{2r}}|v|\leq C\|f\|_{L^{n,1}(B_{2r})}+C\|g\|_{L^{\frac{n}{2},1}(B_{2r})},
\end{equation}
where $C$ depends on $n,\lambda$ and $M$.

Consider now the $\overline{\theta}_{n,\lambda,M}$ from Lemma~\ref{LocalBoundSmallc} and set $\e=\min\{\beta_{n,\lambda,M},\overline{\theta}_{n,\lambda,M}\}$. Then, assuming that $\|c\|_{n,\infty}<\e$ and $\|d\|_{\frac{n}{2},1}<\e$, since $w=u-v$ is a subsolution to $-\dive(A\nabla w+bw)+c\nabla w+dw\leq 0$, \eqref{eq:localLargeBSquare} implies that
\begin{equation}\label{eq:sup2}
\sup_{B_r}w\leq C\left(\fint_{B_{2r}}|w^+|^2\right)^{\frac{1}{2}},
\end{equation}
where $C$ depends on $n,\lambda,\|A\|_{\infty}$ and $M$. Then, \eqref{eq:MoserForB} for $p=2$ follows adding \eqref{eq:sup1} and \eqref{eq:sup2}.

Finally, in the case $p\geq 2$, \eqref{eq:MoserForB} follows from H{\"o}lder's inequality, while in the case $p\in(0,2)$, the proof follows from the argument on \cite[pages 80-82]{Giaquinta}.
\end{proof}

We now turn to the case when $c\in L^{n,q}$ with $q<\infty$ is allowed to have large norm. 

\begin{lemma}\label{LocalBoundSmallb}
Let $A$ be uniformly elliptic and bounded in $B_{2r}$, with ellipticity $\lambda$. Let also $q<\infty$ and $c_1\in L^{n,q}(B_{2r})$ with $\|c_1\|_{n,q}\leq M$.

There exist $\xi=\xi_{n,\lambda}>0$ and $\zeta=\zeta_{n,q,\lambda,M}>0$ such that, if $b\in L^{n,1}(B_{2r})$, $c_2\in L^{n,\infty}(B_{2r})$ and $d\in L^{\frac{n}{2},1}(B_{2r})$ with $\|b\|_{n,1}<\zeta$, $\|c_2\|_{n,\infty}<\xi$ and $\|d\|_{\frac{n}{2},1}<\zeta$, then for any subsolution $u\in W^{1,2}(B_{2r})$ to $-\dive(A\nabla u+bu)+(c_1+c_2)\nabla u+du\leq 0$, we have that
\[
\sup_{B_{r/4}}u\leq C\left(\fint_{B_{2r}}|u^+|^2\right)^{\frac{1}{2}},
\]
where $C$ depends on $n,q,\lambda,\|A\|_{\infty}$ and $M$.
\end{lemma}
\begin{proof}
Let $C_n\geq 1$ be such that $\|h_1+h_2\|_{n,\infty}\leq C_n\|h_1\|_{n,\infty}+C_n\|h_2\|_{n,\infty}$ for all $h_1,h_2\in L^{n,\infty}$ (from \eqref{eq:Seminorm}), and $C_{n,q}\geq 1$ be such that $\|h\|_{n,\infty}\leq C_{n,q}\|h\|_{n,q}$ for all $h\in L^{n,q}$ (from \eqref{eq:LorentzNormsRelations}).
	
Set
\[
\xi_{n,\lambda}=\frac{1}{2C_n}\min\left\{\nu_{n,\lambda},\theta'_{n,\lambda}\right\}>0,
\]
where $\nu_{n,\lambda}$ and $\theta'_{n,\lambda}$ appear in Proposition~\ref{MaxPrincipleC} and Lemma~\ref{localAllSmall}, respectively. For $N\geq 0$, set also $C'_{n,q,\lambda,N}=C_{n,q,\lambda,2^{N/q}C_{n,q}^{-1}\xi_{n,\lambda}}>1$, where the last constant appears in Proposition~\ref{MaxPrincipleB}, and consider the constant $C_0=C_{n,\lambda,\|A\|_{\infty}}\geq 1$ that appears in Lemma~\ref{localAllSmall}.

We claim that, for any integer $N\geq 0$, if $\|c_1\|_{n,q}\leq 2^{N/q}C_{n,q}^{-1}\xi_{n,\lambda}$, then there exists $\zeta_{n,q,\lambda,N}$ such that, if $\|b\|_{n,1}<\zeta_{n,q,\lambda,N}$, $\|c_2\|_{n,\infty}<\xi_{n,\lambda}$ and $\|d\|_{\frac{n}{2},1}<\zeta_{n,q,\lambda,N}$, then
\begin{equation}\label{eq:forInduction2}
\sup_{B_{r/4}}u\leq 8^{\frac{nN}{2}}C_0\prod_{i=0}^NC'_{n,q,\lambda,i}\left(\fint_{B_{2r}}u^2\right)^{\frac{1}{2}}.
\end{equation}
For $N=0$ we can take $\zeta_{n,q,\lambda,0}=\xi_{n,\lambda}$, since we then have that
\[
\|c\|_{n,\infty}\leq C_nC_{n,q}\|c_1\|_{n,q}+C_n\|c_2\|_{n,\infty}\leq 2C_n\xi_{n,\lambda}\leq \theta'_{n,\lambda},
\]
and also $\|b\|_{n,1}\leq\theta'_{n,\lambda}$, $\|d\|_{\frac{n}{2},1}\leq\theta'_{n,\lambda}$, therefore \eqref{eq:forInduction2} for $N=0$ holds from Lemma~\ref{localAllSmall}.

Assume now that \eqref{eq:forInduction2} holds for some $N\geq 0$, and set $\zeta_{n,q,\lambda,{N+1}}=\min\{\zeta_{n,q,\lambda,N},\gamma'_{n,q,\lambda,N+1}\}$, where $\gamma'_{n,q,\lambda,N}=\gamma_{n,q,\lambda,2^{N/q}C_{n,q}^{-1}\xi_{n,\lambda}}$, and the $\gamma$ appears in Proposition~\ref{MaxPrincipleB}. We then continue as in the proof of the Lemma~\ref{LocalBoundSmallc}, using Lemma~\ref{NormDisjoint} for $q>n$ and Proposition~\ref{MaxPrincipleB} instead of Proposition~\ref{MaxPrincipleC}; this shows that \eqref{eq:forInduction2} holds for $N+1$ if $\|c_1\|_{n,q}\leq 2^{(N+1)/q}C_{n,q}^{-1}\xi_{n,\lambda}$, as long as $\|b\|_{n,1}<\zeta_{n,q,\lambda,N+1}$, $\|c_2\|_{n,\infty}<\xi_{n,\lambda}$ and $\|d\|_{\frac{n}{2},1}<\zeta_{n,q,\lambda,N+1}$, and this completes the proof.
\end{proof}

Finally, we add right hand sides and allow different $L^p$ norms.

\begin{prop}\label{MoserC}
Let $A$ be uniformly elliptic and bounded in $B_{2r}$, with ellipticity $\lambda$, and $q<\infty$, $c_1\in L^{n,q}(B_{2r})$ with $\|c_1\|_{n,q}\leq M$. Let also $p>0$ and $f\in L^{n,1}(B_{2r})$, $g\in L^{\frac{n}{2},1}(B_{2r})$.
	
There exist $\xi=\xi_{n,\lambda}>0$ and $\delta=\delta_{n,q,\lambda,M}>0$ such that, if $b\in L^{n,1}(B_{2r})$, $c_2\in L^{n,\infty}(B_{2r})$ and $d\in L^{\frac{n}{2},1}(B_{2r})$ with $\|b\|_{n,1}<\delta$, $\|c_2\|_{n,\infty}<\xi$ and $\|d\|_{\frac{n}{2},1}<\delta$, then for any subsolution $u\in W^{1,2}(B_{2r})$ to $-\dive(A\nabla u+bu)+(c_1+c_2)\nabla u+du\leq -\dive f+g$, we have that
\[
\sup_{B_r}u\leq C\left(\fint_{B_{2r}}|u^+|^p\right)^{\frac{1}{p}}+C\|f\|_{L^{n,1}(B_{2r})}+C\|g\|_{L^{\frac{n}{2},1}(B_{2r})},
\]
where $C$ depends on $n,p,q,\lambda,\|A\|_{\infty}$ and $M$.
\end{prop}
\begin{proof}
The proof is similar to the proof of Proposition~\ref{MoserB}, using Proposition~\ref{MaxPrincipleB} instead of Proposition~\ref{MaxPrincipleC} and Lemma~\ref{LocalBoundSmallb} instead of Lemma~\ref{LocalBoundSmallc}.
\end{proof}

\begin{remark}\label{noSmallness}
Note that the analogue of Propositions~\ref{MoserB} and \ref{MoserC} will hold under no smallness assumptions for $b,d$ or $c,d$ (when $c\in L^{n,q}$, $q<\infty$), but then the constants depend on $b,d$ or $c,d$ and not just on their norms. This can be achieved considering $r'>0$ small enough, so that the norms of $b,d$ or $c,d$ are small enough in all balls of radius $2r'$ that are subsets of $B_{2r}$, and after covering $B_r$ with balls of radius $r'$.
\end{remark}

\subsection{Estimates on the boundary}

We now turn to local boundedness close to the boundary. We will follow the same process as in the case of local boundedness in the interior.

The following are the analogues of \eqref{eq:Cacciopoli} and \eqref{eq:FirstInLemma} close to the boundary; the proof is similar to the one of Lemma~\ref{Coercivity} (as in \cite[proof of Theorem 8.25]{Gilbarg}) and it is omitted.

\begin{lemma}\label{CoercivityBdry}
Let $\Omega\subseteq\bR^n$ be a domain and $B_{2r}\subseteq\bR^n$ be a ball. Let also $A$ be uniformly elliptic and bounded in $\Omega\cap B_{2r}$, with ellipticity $\lambda$.

There exists $\theta=\theta_{n,\lambda}>0$ such that, if $b\in L^{n,1}(\Omega\cap B_{2r})$, $c\in L^{n,\infty}(\Omega\cap B_{2r})$ and $d\in L^{\frac{n}{2},1}(\Omega\cap B_{2r})$ with $\|b\|_{n,1}\leq\theta$, $\|c\|_{n,\infty}\leq\theta$ and $\|d\|_{\frac{n}{2},1}\leq\theta$, then, if $w\in W^{1,2}(\Omega\cap B_{2r})$ is a subsolution to $-\dive(A\nabla w+bw)+c\nabla w+dw\leq 0$ with $w\leq 0$ on $\partial\Omega\cap B_{2r}$, we have that
\[
\int_{\Omega\cap B_r}|\nabla w|^2\leq\frac{C}{r^2}\int_{\Omega\cap B_{2r}}|w^+|^2,
\]
where $C$ depends on $n,\lambda$ and $\|A\|_{\infty}$.

Moreover, for any subsolution $u\in W^{1,2}(\Omega\cap B_{2r})$ to $-\dive(A\nabla u)+c\nabla u\leq 0$ in $\Omega\cap B_{2r}$ and any $\alpha\in(1,2)$, we have that
\[
\sup_{\Omega\cap B_r}u\leq\frac{C}{(\alpha-1)^{n/2}}\left(\fint_{B_{\alpha r}}v^2\right)^{\frac{1}{2}},
\]
where $v=u^+\chi_{\Omega\cap B_{2r}}$, and $C$ depends on $n,\lambda$ and $\|A\|_{\infty}$.
\end{lemma}

To show local boundedness close to the boundary, we will need the following definition from \cite[Theorem 8.25]{Gilbarg}: if $u$ is a function in $\Omega$ and $\partial\Omega\cap B_{2r}\neq\emptyset$, we define
\begin{equation}\label{eq:utildeDfn}
s_u=\sup_{\partial\Omega\cap B_{2r}}u^+,\qquad\tilde{u}(x)=\left\{\begin{array}{l l}\sup\{u(x),s_u\}, & x\in B_{2r}\cap\Omega \\
s_u, & x\in B_{2r}\setminus\Omega\end{array}\right.
\end{equation}
where the supremum over $\partial\Omega\cap B_{2r}$ is defined as on \cite[page 202]{Gilbarg}.

The following proposition concerns the case of large $b$.

\begin{prop}\label{MoserBBoundary}
Let $\Omega\subseteq\bR^n$ be a domain, and $B_{2r}$ be a ball of radius $2r$. Let also $A$ be uniformly elliptic and bounded in $\Omega\cap B_{2r}$, with ellipticity $\lambda$, $b\in L^{n,1}(\Omega\cap B_{2r})$ with $\|b\|_{n,1}\leq M$, and $p>0$, $f\in L^{n,1}(\Omega\cap B_{2r})$, $g\in L^{\frac{n}{2},1}(\Omega\cap B_{2r})$.
	
There exists $\e=\e_{n,\lambda,M}>0$ such that, if $c\in L^{n,\infty}(\Omega\cap B_{2r})$ and $d\in L^{\frac{n}{2},1}(\Omega\cap B_{2r})$ with $\|c\|_{n,\infty}<\e$ and $\|d\|_{\frac{n}{2},1}<\e$, then for any subsolution $u\in W^{1,2}(\Omega\cap B_{2r})$ to $-\dive(A\nabla u+bu)+c\nabla u+du\leq -\dive f+g$, we have that
\[
\sup_{\Omega\cap B_r}\tilde{u}\leq C\left(\fint_{B_{2r}}|\tilde{u}|^p\right)^{\frac{1}{p}}+C\|f\|_{L^{n,1}(\Omega\cap B_{2r})}+C\|g\|_{L^{\frac{n}{2},1}(\Omega\cap B_{2r})},
\]
where $\tilde{u}$ is defined in \eqref{eq:utildeDfn}, and where $C$ depends on $n,p,\lambda,\|A\|_{\infty}$ and $M$.
\end{prop}
\begin{proof}
Subtracting a constant from $u$, and since $\tilde{u}\geq s_u$ in $B_{2r}$, we can reduce to the case when $u\leq 0$ on $\partial\Omega\cap B_{2r}$ (that is, $s_u=0$). Then, based on Lemma~\ref{CoercivityBdry} and \cite[Theorem 8.25]{Gilbarg} instead of Lemma~\ref{Coercivity} and \cite[Theorem 8.17]{Gilbarg}, respectively, we can show the analogue of Lemma~\ref{localAllSmall}, replacing all the balls by their intersections with $\Omega$, for subsolutions $u\in W^{1,2}(\Omega\cap B_{2r})$ with $u\leq 0$ on $\partial\Omega\cap B_{2r}$. We then continue with a similar argument as in the proofs of Lemma~\ref{LocalBoundSmallc} and Proposition~\ref{MoserB}, replacing all the balls by their intersections with $\Omega$.
\end{proof}

Finally, using a similar argument to the above, and going through the arguments of the proofs of Lemma~\ref{LocalBoundSmallb} and Proposition~\ref{MoserC}, we obtain the following estimate close to the boundary, in the case that $c$ is large.

\begin{prop}\label{MoserCBoundary}
Let $\Omega\subseteq\bR^n$ be a domain, and $B_{2r}$ be a ball of radius $2r$. Let also $A$ be uniformly elliptic and bounded in $\Omega\cap B_{2r}$, with ellipticity $\lambda$, and consider $q<\infty$ and $c_1\in L^{n,q}(\Omega\cap B_{2r})$ with $\|c_1\|_{n,q}\leq M$. Let also $p>0$, $f\in L^{n,1}(\Omega\cap B_{2r})$, and $g\in L^{\frac{n}{2},1}(\Omega\cap B_{2r})$.
	
There exist $\xi=\xi_{n,\lambda}>0$ and $\delta=\delta_{n,q,\lambda,M}>0$ such that, if $b\in L^{n,1}(\Omega\cap B_{2r})$, $c_2\in L^{n,\infty}(\Omega\cap B_{2r})$ and $d\in L^{\frac{n}{2},1}(\Omega\cap B_{2r})$ with $\|b\|_{n,1}<\delta$, $\|c_2\|_{n,\infty}<\xi$ and $\|d\|_{\frac{n}{2},1}<\delta$, then for any subsolution $u\in W^{1,2}(\Omega\cap B_{2r})$ to $-\dive(A\nabla u+bu)+(c_1+c_2)\nabla u+du\leq -\dive f+g$, we have that
\[
\sup_{\Omega\cap B_r}\tilde{u}\leq C\left(\fint_{B_{2r}}|\tilde{u}|^p\right)^{\frac{1}{p}}+C\|f\|_{L^{n,1}(\Omega\cap B_{2r})}+C\|g\|_{L^{\frac{n}{2},1}(\Omega\cap B_{2r})},
\]
where $\tilde{u}$ is defined in \eqref{eq:utildeDfn}, and where $C$ depends on $n,p,q,\lambda,\|A\|_{\infty}$ and $M$.
\end{prop}

\begin{remark}\label{noSmallness2}
As in Remark~\ref{noSmallness}, the analogues of Propositions~\ref{MoserBBoundary} and \ref{MoserCBoundary} will hold under no smallness assumptions for $b,d$ or $c,d$ (when $c\in L^{n,q}$, $q<\infty$), with constants depending on $b,d$ or $c,d$ and not just on their norms.
\end{remark}

\section{The reverse Moser estimate and the Harnack inequality}\label{secHarnack}

\subsection{The lower bound}

In order to deduce the Harnack inequality, we will consider negative powers of positive supersolutions to transform them to subsolutions of suitable operators, where the coefficients $b,d$ will be small. This is the context of the following lemma.

\begin{lemma}\label{Pass}
Let $\Omega\subseteq\bR^n$ be a domain, $b,c,f\in L^{n,\infty}(\Omega)$ and $d,g\in L^{\frac{n}{2},\infty}(\Omega)$. Let also $u\in W^{1,2}(\Omega)$ be a supersolution to $-\dive(A\nabla u+bu)+c\nabla u+du\geq -\dive f+g$ with $\inf_{\Omega}u>0$, and consider the function $v=u+\|f\|_{L^{n,1}(\Omega)}+\|g\|_{L^{\frac{n}{2},1}(\Omega)}$. Then, for any $k<0$, $v^k$ is a $W^{1,2}(\Omega)$ subsolution to
\begin{equation}\label{eq:subsol}
-\dive\left(A\nabla(v^k)+\frac{k(bu-f)}{v}v^k\right)+\left(\frac{(k-1)(bu-f)}{v}+c\right)\nabla(v^k)+\frac{k(du-g)}{v}v^k\leq 0.
\end{equation}
\end{lemma}
\begin{proof}
We compute
\[
-\dive(A\nabla(v^k))=-\dive(A\nabla v\cdot kv^{k-1})=-k\dive(A\nabla v)v^{k-1}-k(k-1)A\nabla v\nabla v\cdot v^{k-2}.
\]
From ellipticity of $A$ we have that $A\nabla v\nabla v\geq 0$. Since also $k<0$, the last identity shows that $-\dive(A\nabla(v^k))\leq -\dive(A\nabla u)\cdot kv^{k-1}$. Since $k<0$, $v^{k-1}>0$ and $u$ is a supersolution, we have
\[
-\dive(A\nabla(v^k))\leq (\dive(bu)-c\nabla u-du-\dive f+g)kv^{k-1},
\]
and the proof is complete after a straightforward computation.
\end{proof}

The next lemma bridges the gap between $L^p$ averages for positive and negative $p$.

\begin{lemma}\label{JohnNirenberg}
Let $A$ be uniformly elliptic and bounded in $B_{2r}$, with ellipticity $\lambda$, and $b,c\in L^{n,\infty}(B_{2r})$, $d\in L^{\frac{n}{2},\infty}(B_{2r})$. Let also $u\in W^{1,2}(B_{2r})$ be a supersolution to $-\dive(A\nabla u+bu)+c\nabla u+du\geq 0$ in $B_{2r}$, with $\inf_{B_{2r}}u>0$. Then there exists a constant $a=a_n$ such that
\[
\fint_{B_r}u^a\fint_{B_r}u^{-a}\leq C,
\]
where $C$ depends on $n,\lambda,\|A\|_{\infty}$, $\|b\|_{n,\infty}$, $\|c\|_{n,\infty}$ and $\|d\|_{\frac{n}{2},\infty}$.
\end{lemma}
\begin{proof}
We use the test function from \cite[page 586]{MoserHarnack} (see also \cite[page 195]{Gilbarg}): let $B_{2s}$ be a ball of radius $2s$, contained in $B_{2r}$. If $\phi\geq 0$ be a smooth cutoff supported in $B_{2s}$, with $\phi\equiv 1$ in $B_s$ and $|\nabla\phi|\leq\frac{C}{s}$, then the function $\phi^2u^{-1}$ is nonnegative and belongs to $W_0^{1,2}(B_{2s})$. Hence, using it as a test function, we obtain that
\[
\int_{B_{2s}}\left(A\nabla u\frac{2\phi\nabla\phi}{u}-A\nabla u\frac{\phi^2\nabla u}{u^2}+b\frac{2\phi\nabla\phi}{u}u-b\frac{\phi^2\nabla u}{u^2}u+c\nabla u\frac{\phi^2}{u}+du\frac{\phi^2}{u}\right)\geq 0,
\]
hence
\[
\int_{B_{2s}}A\nabla u\frac{\nabla u}{u^2}\phi^2\leq\int_{B_{2r}}\left(A\nabla u\frac{2\phi\nabla\phi}{u}+2b\nabla\phi\cdot\phi-b\frac{\nabla u}{u}\phi^2+c\nabla u\frac{\phi^2}{u}+d\phi^2\right).
\]
Using ellipticity of $A$, the Cauchy-Schwartz inequality, and Cauchy's inequality with $\e$, we obtain
\begin{align*}
\int_{B_{2s}}\frac{|\nabla u|^2}{u^2}\phi^2&\leq C\int_{B_{2s}}\left(|\nabla\phi|^2+|b\nabla\phi|\phi+(|b|^2+|c|^2+|d|)\phi^2\right)\\
&\leq Cs^{n-2}+Cs^{-1}\|b\|_{n,\infty}\|1\|_{L^{\frac{n}{n-1},1}(B_{2s})}+C\left\||b|^2+|c|^2+|d|\right\|_{\frac{n}{2},\infty}\|1\|_{L^{\frac{n}{n-2},1}(B_{2s})}\\
&\leq Cs^{n-2},
\end{align*}
where $C$ depends on $n,\lambda,\|A\|_{\infty}$, $\|b\|_{n,\infty}$, $\|c\|_{n,\infty}$ and $\|d\|_{\frac{n}{2},\infty}$, and where we used \eqref{eq:Holder} for the second estimate. The proof is complete using the Poincar{\'e} inequality and the John-Nirenberg inequality, as on \cite[page 586]{MoserHarnack}.
\end{proof}

The next bound is a reverse Moser estimate for supersolutions. Surprisingly, if we assume that the coefficient $c$ belongs to $L^{n,q}$ for some $q<\infty$, then we obtain a scale invariant estimate with ``good" constants under no smallness assumption on the coefficients. As mentioned before, for the Moser estimate in Propositions~\ref{MoserB} and \ref{MoserC}, such a bound cannot hold with ``good" constants under these assumptions.

\begin{prop}\label{lowerBound}
Let $A$ be uniformly elliptic and bounded in $B_{2r}$, with ellipticity $\lambda$. Let also $b,f\in L^{n,1}(B_{2r})$, $c_1\in L^{n,q}(B_{2r})$ for some $q<\infty$, and $d,g\in L^{\frac{n}{2},1}(B_{2r})$, with $\|b\|_{n,1}\leq M_b$, $\|c_1\|_{n,q}\leq M_c$ and $\|d\|_{\frac{n}{2},1}\leq M_d$.

There exist $a=a_n>0$ and $\xi=\xi_{n,\lambda}>0$ such that, if $c_2\in L^{n,\infty}(B_{2r})$ with $\|c_2\|_{n,\infty}<\xi$, then for any nonnegative supersolution $u\in W^{1,2}(B_{2r})$ to $-\dive(A\nabla u+bu)+(c_1+c_2)\nabla u+du\geq -\dive f+g$, we have that
\[
\left(\fint_{B_r}u^a\right)^{\frac{1}{a}}\leq C\inf_{B_{r/2}}u+C\|f\|_{L^{n,1}(\Omega\cap B_{2r})}+C\|g\|_{L^{\frac{n}{2},1}(\Omega\cap B_{2r})},
\]
where $C$ depends on $n,q,\lambda,\|A\|_{\infty}$, $M_b,M_c$ and $M_d$.
\end{prop}
\begin{proof}
Adding a constant $\delta>0$ to $u$, we may assume that $\inf_{B_{2r}}u>0$; the general case will follow by letting $\delta\to 0$. Set $v=u+\|f\|_{L^{n,1}(B_{2r})}+\|g\|_{L^{\frac{n}{2},1}(B_{2r})}$, then $v$ is a supersolution to
\[
-\dive\left(A\nabla v+\frac{bu-f}{v}v\right)+c\nabla v+\frac{du-g}{v}v\geq 0,
\]
with
\[
\left\|\frac{bu-f}{v}\right\|_{n,1}\leq C_n\|b\|_{n,1}+C_n,\qquad \left\|\frac{du-g}{v}\right\|_{\frac{n}{2},1}\leq C_n\|d\|_{\frac{n}{2},1}+C_n.
\]
Then, since $\inf_{B_{2r}}v>0$, Lemma~\ref{JohnNirenberg} implies that there exists $a=a_n$ such that
\begin{equation}\label{eq:alpha}
\fint_{B_r}v^a\fint_{B_r}v^{-a}\leq C,
\end{equation}
where $C$ depends on $n,q,\lambda,\|A\|_{\infty}$, $M_b,M_c$ and $M_d$.

For $k\in(-1,0)$ to be chosen later, $v^k$ is a $W^{1,2}(B_{2r})$ subsolution to \eqref{eq:subsol} for $c=c_1+c_2$, and
\[
\left\|\frac{(k-1)(bu-f)}{v}+c_1\right\|_{n,q}\leq C_{n,q}(1-k)\left\|\frac{bu}{v}\right\|_{n,q}+C_{n,q}(1-k)\left\|\frac{f}{v}\right\|_{n,q}+C_{n,q}\|c_1\|_{n,q}\leq M,
\]
where $M$ depends on $n,q,M_b$ and $M_c$. Then, for the $\xi_{n,\lambda}$ and the $\delta_{n,q,\lambda,M}>0$ from Proposition~\ref{MoserC} and \eqref{eq:subsol}, if
\begin{equation}\label{eq:smallk}
\left\|\frac{k(bu-f)}{v}\right\|_{L^{n,1}(B_r)}<\delta_{n,q,\lambda,M},\quad\|c_2\|_{L^{n,\infty}(B_r)}<\xi_{n,\lambda},\quad\left\|\frac{k(du-g)}{v}\right\|_{L^{\frac{n}{2},1}(B_r)}<\delta_{n,q,\lambda,M},
\end{equation}
then $v^k$ satisfies the estimate
\[
\sup_{B_{r/2}}v^k\leq C\fint_{B_r}v^k,
\]
where $C$ depends on $n,q,\lambda,\|A\|_{\infty}$ and $M$. It is true that \eqref{eq:smallk} holds for some $k\in(-a,0)$, depending on $n,q,\lambda,M_b,M_c$ and $M_d$; hence, for this $k$,
\begin{equation}\label{eq:plugk}
\left(\fint_{B_r}v^k\right)^{\frac{1}{k}}\leq C(\sup_{B_{r/2}}v^k)^{\frac{1}{k}}=C\inf_{B_{r/2}}v,
\end{equation}
where $C$ depends on $n,q,\lambda,\|A\|_{\infty},M_b,M_c$ and $M_d$. Since $-\frac{a}{k}>1$, H{\"o}lder's inequality implies that
\[
\fint_{B_r}v^k\leq\left(\fint_{B_r}v^{-a}\right)^{-\frac{k}{a}}\Rightarrow \left(\fint_{B_r}v^k\right)^{\frac{1}{k}}\geq\left(\fint_{B_r}v^{-a}\right)^{-\frac{1}{a}}\geq C\left(\fint_{B_r}v^a\right)^{\frac{1}{a}},
\]
where we used \eqref{eq:alpha} for the last step. Then, plugging the last estimate in \eqref{eq:plugk}, and using the definition of $v$, the proof is complete.
\end{proof}

\subsection{Estimates on the boundary}

We now consider the analogue of Proposition~\ref{lowerBound} close to the boundary. We will need the analogue of the definition of $\tilde{u}$ in \eqref{eq:utildeDfn}, from \cite[Theorem 8.26]{Gilbarg}: if $u\geq 0$ is a function in $\Omega$ and $\partial\Omega\cap B_{2r}\neq\emptyset$, we define
\begin{equation}\label{eq:ubarDfn}
m_u=\inf_{\partial\Omega\cap B_{2r}}u,\qquad\bar{u}(x)=\left\{\begin{array}{l l}\inf\{u(x),m_u\}, & x\in B_{2r}\cap \Omega \\
m_u, & x\in B_{2r}\setminus\Omega\end{array}\right..
\end{equation}

The following is the analogue of Lemma~\ref{JohnNirenberg} close to the boundary.

\begin{lemma}\label{JohnNirenbergBdry}
Let $A$ be uniformly elliptic and bounded in $\Omega\cap B_{2r}$, with ellipticity $\lambda$, and $b,c\in L^{n,\infty}(\Omega\cap B_{2r})$, $d\in L^{\frac{n}{2},\infty}(\Omega\cap B_{2r})$. Let also $u\in W^{1,2}(B_{2r})$ be a nonnegative supersolution to $-\dive(A\nabla u+bu)+c\nabla u+du\geq 0$ in $B_{2r}$, and consider the function $\bar{u}$ from \eqref{eq:ubarDfn}. If $\inf_{\Omega\cap B_{2r}}u>0$ and $m_u>0$, then there exists a constant $a=a_n$ such that
\[
\fint_{B_r}\bar{u}^a\fint_{B_r}\bar{u}^{-a}\leq C,
\]
where $C$ depends on $n,\lambda,\|A\|_{\infty}$, $\|b\|_{n,\infty}$, $\|c\|_{n,\infty}$ and $\|d\|_{\frac{n}{2},\infty}$.
\end{lemma}
\begin{proof}
As in the proof of \cite[Theorem 8.26]{Gilbarg}, set $v=\bar{u}^{-1}-m_u^{-1}\in W^{1,2}(\Omega\cap B_{2r})$, which is nonnegative in $\Omega\cap B_{2r}$ and vanishes on $\partial\Omega\cap B_{2r}$. Then, considering the test function $v\phi^2$, where $\phi$ is a suitable cutoff function, and using that $v>0$ if and only if $\bar{u}=u$, the proof follows by an argument as in the proof of Lemma~\ref{JohnNirenberg}.
\end{proof}

Using the previous lemma, we can show the following estimate.

\begin{prop}\label{lowerBoundBoundary}
Let $\Omega\subseteq\bR^n$ be a domain, $B_{2r}$ be a ball of radius $2r$, and let $A$ be uniformly elliptic and bounded in $\Omega\cap B_{2r}$, with ellipticity $\lambda$. Let also $b,f\in L^{n,1}(\Omega\cap B_{2r})$, $c_1\in L^{n,q}(\Omega\cap B_{2r})$ for some $q<\infty$, and $d,g\in L^{\frac{n}{2},1}(\Omega\cap B_{2r})$, with $\|b\|_{n,1}\leq M_b$, $\|c_1\|_{n,q}\leq M_c$ and $\|d\|_{\frac{n}{2},1}\leq M_d$.
	
There exist $a=a_n>0$ and $\xi=\xi_{n,\lambda}>0$ such that, if $c_2\in L^{n,\infty}(\Omega\cap B_{2r})$ with $\|c_2\|_{n,\infty}<\xi$, then for any nonnegative supersolution $u\in W^{1,2}(\Omega\cap B_{2r})$ to $-\dive(A\nabla u+bu)+(c_1+c_2)\nabla u+du\geq -\dive f+g$, we have that
\[
\left(\fint_{B_r}\bar{u}^a\right)^{\frac{1}{a}}\leq C\inf_{B_{r/2}}\bar{u}+C\|f\|_{L^{n,1}(\Omega\cap B_{2r})}+C\|g\|_{L^{\frac{n}{2},1}(\Omega\cap B_{2r})},
\]
where $\bar{u}$ is defined in \eqref{eq:ubarDfn}, and where $C$ depends on $n,q,\lambda,\|A\|_{\infty}$, $M_b,M_c$ and $M_d$.
\end{prop}
\begin{proof}
As in the proof of Proposition~\ref{lowerBound}, we can assume that $\inf_{\Omega\cap B_{2r}}u>0$, $m_u>0$, and $f,g\equiv 0$. Let $a=a_n$ be as in Lemma~\ref{JohnNirenbergBdry}. Then, Lemma~\ref{Pass} and Proposition~\ref{MoserCBoundary} show that, for suitable $k\in(-a,0)$, if $w_k=u^k$ and $\tilde{w_k}$ is as in \eqref{eq:utildeDfn}, we have that
\[
\sup_{\Omega\cap B_{r/2}}\tilde{w_k}\leq C\fint_{B_r}\tilde{w_k}\leq C\left(\fint_{B_r}\tilde{w_k}^{-\frac{a}{k}}\right)^{-\frac{k}{a}}.
\]
Since $\tilde{w_k}=\bar{u}^k$, the proof is complete using also Lemma~\ref{JohnNirenbergBdry}.
\end{proof}

\subsection{The Harnack inequality, and local continuity}

We now show the Harnack inequality in the cases when $b,d$ are small, or when $c,d$ are small.

\begin{thm}\label{HarnackForB}
Let $A$ be uniformly elliptic and bounded in $B_{2r}$, with ellipticity $\lambda$. Let also $b,f\in L^{n,1}(B_{2r})$ with $\|b\|_{n,1}\leq M$, and $g\in L^{\frac{n}{2},1}(B_{2r})$. 

There exists $\e_{n,\lambda,M}>0$ such that, if $c\in L^{n,\infty}(B_{2r})$ and $d\in L^{\frac{n}{2},1}(B_{2r})$ with $\|c\|_{n,\infty}<\e$ and $\|d\|_{\frac{n}{2},1}<\e$, then for any nonnegative solution $u\in W^{1,2}(B_{2r})$ to $-\dive(A\nabla u+bu)+c\nabla u+du =-\dive f+g$, we have that
\[
\sup_{B_r}u\leq C\inf_{B_r}u+C\|f\|_{L^{n,1}(B_{2r})}+C\|g\|_{L^{\frac{n}{2},1}(B_{2r})},
\]
where $C$ depends on $n,\lambda,\|A\|_{\infty}$ and $M$.
\end{thm}
\begin{proof}
The proof is a combination of Proposition~\ref{MoserB} (choosing $p=a_n$ in \eqref{eq:MoserForB}, as in Proposition~\ref{lowerBound}), and Proposition~\ref{lowerBound}, (considering $q=n$ and $c_1\equiv 0$), after also covering $B_r$ with balls of radius $r/4$.
\end{proof}

\begin{thm}\label{HarnackForC}
Let $A$ be uniformly elliptic and bounded in $B_{2r}$, with ellipticity $\lambda$, and $q<\infty$, $c_1\in L^{n,q}(B_{2r})$ with $\|c_1\|_{n,q}\leq M$. Let also $f\in L^{n,1}(B_{2r})$, $g\in L^{\frac{n}{2},1}(B_{2r})$.
	
There exist $\xi=\xi_{n,\lambda}>0$ and $\delta=\delta_{n,q,\lambda,M}>0$ such that, if $b\in L^{n,1}(B_{2r})$, $c_2\in L^{n,\infty}(B_{2r})$ and $d\in L^{\frac{n}{2},1}(B_{2r})$ with $\|b\|_{n,1}<\delta$, $\|c_2\|_{n,\infty}<\xi$ and $\|d\|_{\frac{n}{2},1}<\delta$, then for any nonnegative solution $u\in W^{1,2}(B_{2r})$ to $-\dive(A\nabla u+bu)+c\nabla u+du=-\dive f+g$, we have that
\[
\sup_{B_r}u\leq C\inf_{B_r}u+C\|f\|_{L^{n,1}(B_{2r})}+C\|g\|_{L^{\frac{n}{2},1}(B_{2r})},
\]
where $C$ depends on $n,q,\lambda,\|A\|_{\infty}$ and $M$.
\end{thm}
\begin{proof}
The proof follows by a combination of Propositions~\ref{MoserC} and \ref{lowerBound}.
\end{proof}

We now turn to local continuity of solutions. For the following theorem, for $\rho\leq 2r$, we set
\begin{equation}\label{eq:QDfn}
Q_{b,d}(\rho)=\sup\left\{\|b\|_{L^{n,1}(B'_{\rho})}+\|d\|_{L^{\frac{n}{2},1}(B'_{\rho})}:B'_{\rho}\subseteq B_{2r}\right\},
\end{equation}
where $B_{\rho}'$ runs over all the balls of radius $\rho$ that are subsets of $B_{2r}$. Also, we will follow the argument on \cite[pages 200-202]{Gilbarg}.

\begin{thm}\label{ContinuityB}
Let $A$ be uniformly elliptic and bounded in $B_{2r}$, with ellipticity $\lambda$. Let also $b,f\in L^{n,1}(B_{2r})$ with $\|b\|_{n,1}\leq M$, $g\in L^{\frac{n}{2},1}(B_{2r})$, and $\mu\in(0,1)$. 

For every $\mu\in(0,1)$, there exists $\e=\e_{n,\lambda,M}>0$ and $\alpha=\alpha_{n,\lambda,\|A\|_{\infty},M,\mu}\in(0,1)$ such that, if $c\in L^{n,\infty}(B_{2r})$ and $d\in L^{\frac{n}{2},1}(B_{2r})$ with $\|c\|_{n,\infty}<\e$ and $\|d\|_{\frac{n}{2},1}<\e$, then for any solution $u\in W^{1,2}(B_{2r})$ to $-\dive(A\nabla u+bu)+c\nabla u+du =-\dive f+g$, we have that
\begin{multline*}
|u(x)-u(y)|\leq C\left(\frac{|x-y|^{\alpha}}{r^{\alpha}}+Q_{b,d}(|x-y|^{\mu}r^{1-\mu})\right)\left(\fint_{B_{2r}}|u|+Q_{f,g}(2r)\right)+CQ_{f,g}(|x-y|^{\mu}r^{1-\mu}),
\end{multline*}
for any $x,y\in B_r$, where $Q$ is defined in \eqref{eq:QDfn} and $C$ depends on $n,\lambda,\|A\|_{\infty}$ and $M$.	
\end{thm}
\begin{proof}
Let $\rho\in(0,r]$, and set $M(\rho)=\sup_{B_{\rho}}u$, $m(\rho)=\inf_{B_{\rho}}u$. Then $v_1=M(\rho)-u$ is nonnegative in $B_{\rho}$, and solves the equation
\[
-\dive(A\nabla v_1+bv_1)+c\nabla v_1+dv_1=-\dive(M(\rho)b-f)+(M(\rho) d-g)
\]
in $B_{\rho}$. 
Hence, from Theorem~\ref{HarnackForB}, \eqref{eq:Seminorm} and \eqref{eq:QDfn}, we obtain that
\begin{align}\label{eq:up1}
\begin{split}
M(\rho)-m\left(\frac{\rho}{2}\right)&=\sup_{B_{\rho/2}}v_1\leq C\inf_{B_{\rho/2}}v_1+C\|M(\rho)b-f\|_{L^{n,1}(B_{\rho})}+C\|M(\rho) d-g\|_{L^{\frac{n}{2},1}(B_{\rho})}\\
&=C\left(M(\rho)-M\left(\frac{\rho}{2}\right)\right)+C\sup_{B_r}|u|\cdot Q_{b,d}(\rho)+CQ_{f,g}(\rho),
\end{split}
\end{align}
where $C$ depends on $n,\lambda,\|A\|_{\infty}$ and $M$. Moreover, $v_2=u-m(\rho)$ is nonnegative in $B_{\rho}$, and solves the equation
\[
-\dive(A\nabla v_2+bv_2)+c\nabla v_2+dv_2=-\dive(f-m(\rho)b)+(g-m(\rho)d)
\]
in $B_{\rho}$. Hence, from Theorem~\ref{HarnackForB}, as in \eqref{eq:up1},
\begin{equation}\label{eq:up2}
M\left(\frac{\rho}{2}\right)-m(\rho)\leq C\left(m\left(\frac{\rho}{2}\right)-m(\rho)\right)+C\sup_{B_r}|u|\cdot Q_{b,d}(\rho)+CQ_{f,g}(\rho).
\end{equation}
Adding \eqref{eq:up1} and \eqref{eq:up2} and defining $\omega(\rho)=M(\rho)-m(\rho)$, we obtain that
\[
\omega\left(\frac{\rho}{2}\right)\leq\theta_0\omega(\rho)+C\sup_{B_r}|u|\cdot Q_{b,d}(\rho)+CQ_{f,g}(\rho),
\]
where $\theta_0=\frac{C-1}{C+1}\in(0,1)$. Then, \cite[Lemma 8.23]{Gilbarg} shows that, for $\rho\leq r$,
\[
\omega(\rho)\leq C\frac{\rho^{\alpha}}{r^{\alpha}}\omega(r)+C\sup_{B_r}|u|\cdot Q_{b,d}(\rho^{\mu}r^{1-\mu})+CQ_{f,g}(\rho^{\mu}r^{1-\mu}),
\]
where $C$ depends on $n,\lambda,\|A\|_{\infty},M$, and $\alpha=\alpha_{n,\lambda,\|A\|_{\infty},M,\mu}$. We then bound $\sup_{B_r}|u|$ using Proposition~\ref{MoserB} (applied to $u$ and $-u$, for $p=1$), which completes the proof.
\end{proof}

Finally, based on Proposition~\ref{MoserC} and Theorem~\ref{HarnackForC}, we obtain the following theorem when $b,d$ are small.

\begin{thm}\label{ContinuityC}
Let $A$ be uniformly elliptic and bounded in $B_{2r}$, with ellipticity $\lambda$, and $q<\infty$, $c_1\in L^{n,q}(B_{2r})$ with $\|c_1\|_{n,q}\leq M$. Let also $f\in L^{n,1}(B_{2r})$, $g\in L^{\frac{n}{2},1}(B_{2r})$.

For every $\mu\in(0,1)$, there exist $\xi=\xi_{n,\lambda}>0$, $\delta=\delta_{n,q,\lambda,M}>0$ and $\alpha=\alpha_{n,\lambda,\|A\|_{\infty},M,\mu}$ such that, if $b\in L^{n,1}(B_{2r})$, $c_2\in L^{n,\infty}(B_{2r})$ and $d\in L^{\frac{n}{2},1}(B_{2r})$ with $\|b\|_{n,1}<\delta$, $\|c_2\|_{n,\infty}<\xi$ and $\|d\|_{\frac{n}{2},1}<\delta$, then for any solution $u\in W^{1,2}(B_{2r})$ to $-\dive(A\nabla u+bu)+c\nabla u+du=-\dive f+g$, we have that
\begin{multline*}
|u(x)-u(y)|\leq C\left(\frac{|x-y|^{\alpha}}{r^{\alpha}}+Q_{b,d}(|x-y|^{\mu}r^{1-\mu})\right)\cdot\left(\fint_{B_{2r}}|u|+Q_{f,g}(2r)\right)+CQ_{f,g}(|x-y|^{\mu}r^{1-\mu}),
\end{multline*}
for any $x,y\in B_r$, where $Q$ is defined in \eqref{eq:QDfn} and $C$ depends on $n,q,\lambda,\|A\|_{\infty}$ and $M$.
\end{thm}

\begin{remark}\label{noSmallness3}
As in Remarks~\ref{noSmallness} and \ref{noSmallness2}, the analogues of Theorems~\ref{HarnackForB} - \ref{ContinuityC} will hold under no smallness assumptions for $b,d$ and $c,d$ (when $c\in L^{n,q}$, $q<\infty$), but then the constants depend on $b,d$ or $c,d$ and not just on their norms.
\end{remark}

\section{Optimality of the assumptions}\label{secOptimality}

We now turn to showing that our assumptions are optimal in order to deduce the estimates we have shown so far, in the setting of Lorentz spaces. We first show optimality for $b$ and $d$.

\begin{remark}\label{optimalB}
Considering the operators $\mathcal{L}_1u=-\Delta u-\dive(bu)$ and $\mathcal{L}_2u=-\Delta u+du$, an assumption of the form $b\in L^{n,q}$, $d\in L^{\frac{n}{2},q}$ for some $q>1$, with $\|b\|_{n,q}$, $\|d\|_{\frac{n}{2},1}$ being as small as we want, is not enough to guarantee the pointwise bounds in the maximum principle and Moser's estimate. Indeed, as in Lemma \cite[Lemma 7.4]{KimSak}, set $u_{\delta}(x)=\left(-\ln|x|\right)^{\delta}$ and $b_{\delta}(x)=-\frac{\delta x}{|x|^2\ln|x|}$. Then, for $\delta\in(-1,1)$, $b\in L^{n,q}(B_{1/e})$ for all $q>1$, $u_{\delta}\in W^{1,2}(B_{1/e})$, and $u_{\delta}$ solves the equation
\[
-\Delta u-\dive(b_{\delta}u_{\delta})=0
\]
in $B_{1/e}$. However, $v_{\delta}\equiv 1$ on $\partial B_{1/e}$, and $v_{\delta}\to\infty$ as $|x|\to 0$ for $\delta>0$, so the assumption $b\in L^{n,1}$ is optimal for the maximum principle and the Moser estimate. Note that $u_{\delta}$ also solves the equation
\[
-\Delta u_{\delta}+d_{\delta}u_{\delta}=0,\qquad d_{\delta}(x)=\frac{\delta(\delta-1)}{|x|^2\ln^2|x|}+\frac{\delta(n-2)}{|x|^2\ln|x|},
\]
and $d_{\delta}\in L^{\frac{n}{2},q}(B_{1/e})$ for every $q>1$; hence, the assumption $d\in L^{\frac{n}{2},1}$ is again optimal.

The same functions $b_{\delta}$ and $d_{\delta}$ serve as counterexamples to show optimality for the spaces of $b,d$ in the reverse Moser estimate. In particular, considering $\delta<0$, we have that $u_{\delta}(0)=0$, while $u_{\delta}$ does not identically vanish close to $0$, therefore the reverse Moser estimate cannot hold.
\end{remark}

We now turn to optimality for smallness of $c$, when $c\in L^{n,\infty}$.

\begin{remark}
In the case of the operator $\mathcal{L}_0u=-\Delta u+c\nabla u$ with $c\in L^{n,\infty}$, smallness in norm is a necessary condition, in order to obtain all the estimates we have considered. Indeed, if $u(x)=-\ln|x|-1$, then $u\in W_0^{1,2}(B_{1/e})$, and $u$ solves the equation
\[
-\Delta u+c\nabla u=0,\qquad c=\frac{(2-n)x}{|x|^2}\in L^{n,\infty}(B_{1/e}).
\]
However, $u$ is not bounded in $B_{1/e}$, so the maximum principle, as well as Moser's and Harnack's estimates fail. On the other hand, the function $v(x)=(-\ln|x|)^{-1}\in W_0^{1,2}(B_{1/e})$ solves the equation
\[
-\Delta v+c'\nabla v=0,\qquad c'=\frac{(n-2)x}{|x|^2}-\frac{2x}{|x|^2\ln|x|}\in L^{n,\infty}(B_{1/e}),
\]
with $v(0)=0$ and $v$ not identically vanishing close to $0$, therefore smallness for $c\in L^{n,\infty}$ in the reverse Harnack estimate is necessary.
\end{remark}

Finally, we show the optimality of the assumption that either $b,d$ should be small, or $c,d$ should be small, so that in the maximum principle, as well as Moser's and Harnack's estimates, the constants depend only on the norms of the coefficients. The fact that $d$ should be small is based on the following construction.

\begin{prop}\label{dShouldBeSmall}
There exists a bounded sequence $(d_N)$ in $L^{\frac{n}{2},1}(B_1)$ and a sequence $(u_N)$ of nonnegative $W_0^{1,2}(B_1)\cap C(\overline{B})$ functions such that, for all $N\in\mathbb N$, $u_N$ is a solution to the equation $-\Delta u_N+d_Nu_N=0$ in $B_1$, and
\[
\|u_N\|_{W_0^{1,2}(B_1)}\leq C,\quad\text{while}\quad u_N(0)\xrightarrow[N\to\infty]{}\infty.
\]
\end{prop}
\begin{proof}
We define
\[
v(r)=\left\{\begin{array}{l l}
\frac{n}{2}+\left(1-\frac{n}{2}\right)r^2, & 0<r\leq 1 \\
r^{2-n}, & r>1.\end{array}\right.
\]
Set $u(x)=v(|x|)$, then it is straightforward to check that $u$ is radially decreasing, $u\geq 1$ in $B_1$, $u\leq\frac{n}{2}$ in $\bR^n$, and $u\in Y^{1,2}(\bR^n)\cap C^1(\bR^n)$. Then, the function $d=n(2-n)u^{-1}\chi_{B_1}$ is bounded and supported in $B_1$, and $u$ is a solution to the equation $-\Delta u+du=0$ in $\bR^n$.	
	
We now let $N\in\mathbb N$ with $N\geq 2$, and set $B_N$ to be the ball of radius $N$, centered at $0$. We will modify $u$ to be a $W_0^{1,2}(B_N)$ solution to a slightly different equation: for this, set $w_N=u-v(N)$, and also
\[
d_N=\frac{du}{u-v(N)}.
\]
Since $d$ is supported in $B_1$, $d_N$ is well defined. Note also that $w_N\in W_0^{1,2}(B_N)$, and $w_N$ is a solution to the equation $-\Delta w_N+d_Nw_N=0$ in $B_N$. Moreover, since $d$ is supported in $B_1$, $u\geq 1$ in $B_1$ and $v$ is decreasing, we have that
\[
\|d_N\|_{L^{\frac{n}{2},1}(B_N)}\leq C_n\|d_N\|_{L^{\infty}(B_1)}\leq C_n\frac{\|d\|_{L^{\infty}(B_1)}\|u\|_{L^{\infty}(B_1)}}{1-v(N)}\leq C_n.
\]
Let now $\tilde{d}_N(x)=N^2d_N(Nx)$ and $\tilde{w}_N(x)=w_N(Nx)$, for $x\in B_1$. Then $\tilde{w}_N\in W_0^{1,2}(B_1)$, $(\tilde{d}_N)$ is bounded in $L^{\frac{n}{2},1}(B_1)$, and $\tilde{w}_N$ is a solution to the equation $-\Delta\tilde{w}_N+\tilde{d}_N\tilde{w}_N=0$ in $B_1$. Moreover, $\tilde{w}_N(0)\geq C_n$, while
\[
\int_{B_1}|\nabla\tilde{w}_N|^2=N^{2-n}\int_{B_N}|\nabla w_N|^2= N^{2-n}\int_{B_N}|\nabla u|^2\xrightarrow[N\to\infty]{}0,
\]
since $\nabla u\in L^2(\bR^n)$. Hence, considering the function $\frac{\tilde{w}_N}{\|\nabla\tilde{w}_N\|_{L^2(B_1)}}$ completes the proof.
\end{proof}

\begin{remark}\label{bcShouldBeSmall}
If $d_N,u_N$ are as in Proposition~\ref{dShouldBeSmall}, then using the functions $e_N$ from Lemma~\ref{Reduction} that solve the equation $\dive e_N=d_N$ in $B_1$, we have that
\[
-\dive(\nabla u_N-e_Nu)-e_N\nabla u_N=0.
\]
So, for the operator $\mathcal{L}u=-\dive(A\nabla u+bu)+c\nabla u$, if both $b,c$ are allowed to be large, then the conclusion of Proposition~\ref{dShouldBeSmall} shows that the constants in the maximum principle, as well as Moser's and Harnack's estimates, cannot depend only on the norms of the coefficients.
\end{remark}

\bibliographystyle{amsalpha}
\bibliography{Bibliography}

\end{document}